\documentclass[12pt]{amsart}
\usepackage{amssymb,tikz}
\usepackage{verbatim}
\usepackage{graphicx}
\usepackage[cmtip,all]{xy}
\usepackage{amsmath}
\usepackage{amsfonts,amssymb,amsthm}
\usepackage{times}
\usepackage{amscd}
\usepackage{epsfig}
\usepackage{float}

\usepackage{pb-diagram}

\usepackage{srctex}
\definecolor{refkey}{gray}{.5}   
\definecolor{labeled}{gray}{.5} 
\newtheorem{theo}{Theorem}[section]
\newtheorem{thm}[theo]{Theorem}
\newtheorem{lem}[theo]{Lemma}
\newtheorem{cor}[theo]{Corollary}

\newtheorem*{thm*}{Theorem}

\theoremstyle{definition}
\newtheorem{dfn}[theo]{Definition}
\newtheorem{exa}[theo]{Example}

\theoremstyle{remark}

\numberwithin{equation}{section}

\def\L{\mathbb{L}}
\newcommand{\ph}{\varphi}
\def\dc{\mathcal{D}}
\newcommand{\cri}{\mathrm{Cr}}
\newcommand{\di}{\mathrm{Di}}
\newcommand{\ta}{\theta}
\newcommand{\hc}{\mbox{$\mathbb{\widehat{C}}$}}
\newcommand{\iU}{U^\infty}
\newcommand{\ql}{quadrilateral{}}

\newcommand{\ca}{\mathrm{CA}}

\newcommand{\C}{\mathbb{C}}

\newcommand{\disk}{\mathbb{D}}

\newcommand{\cdisk}{\ol{\mathbb{D}}}
\newcommand{\bbd}{\mathbb{D}}

\newcommand{\ol}{\overline}

\newcommand{\sm}{\setminus}

\newcommand{\n}{\ol{n}}
\newcommand{\hell}{\hat{\ell}}

\newcommand{\qml}{\mathrm{QML}}
\newcommand{\bd}{\mathrm{Bd}}

\newcommand{\lam}{\mathcal{L}}
\newcommand{\hlam}{\widehat{\mathcal{L}}}

\newcommand{\ch}{\mathrm{CH}}

\newcommand{\si}{\sigma}
\newcommand{\0}{\varnothing}
\newcommand{\uc}{\mathbb{S}}

\newcommand{\al}{\alpha}

\newcommand{\be}{\beta}

\newcommand{\M}{\mathcal{M}}

\newcommand{\laq}{\mathbb{L}^q}
\newcommand{\olq}{\ol{\mathbb{L}^q_2}}

\renewcommand{\le}{\leqslant}

\begin{document}
\date{July 17, 2017}

\title[Perfect subspaces of quadratic laminations]
{Perfect subspaces of quadratic laminations}

\dedicatory{Dedicated to the memory of Tan Lei}

\author{Alexander~Blokh}

\address[Alexander~Blokh]
{Department of Mathematics\\ University of Alabama at Birmingham\\
Birmingham, AL 35294}

\email{ablokh@math.uab.edu}


\author{Lex Oversteegen}

\address[Lex Oversteegen]
{Department of Mathematics\\ University of Alabama at Birmingham\\
Birmingham, AL 35294}

\email{overstee@uab.edu}


\author{Vladlen~Timorin}

\address[V.~Timorin]{National Research University Higher School of Economics\\
6 Usacheva St., 119048 Moscow, Russia}

\thanks{The third named author was partially supported by the Russian Academic Excellence Project '5-100'}

\subjclass[2010]{Primary 37F20; Secondary 37F10, 37F50}

\keywords{Complex dynamics; laminations; Mandelbrot set; Julia set}

\begin{abstract}
The combinatorial Mandelbrot set is a continuum in the plane, whose boundary can be defined, up to a homeomorphism,
  as the quotient space of the unit circle by an explicit equivalence relation.
This equivalence relation was described by Douady and, in different terms, by Thurston.
Thurston used quadratic invariant laminations as a major tool.
As has been previously shown by the authors, the combinatorial Mandelbrot set can be interpreted as a quotient of the space of all limit quadratic invariant laminations.
The topology in the space of laminations is defined by the Hausdorff distance.
In this paper, we describe two similar quotients.
In the first case, the identifications are the same but the space is smaller than that taken for the Mandelbrot set.
The result (the quotient space) is obtained from the Mandelbrot set by ``unpinching'' the transitions between adjacent hyperbolic components.
In the second case, we do not identify non-renormalizable laminations while
identifying renormalizable laminations according to which hyperbolic lamination
they can be ``unrenormalised'' to.

\end{abstract}

\maketitle

\section*{Introduction}\label{s:intro}

Understanding families of one-dimensional complex polynomials is a central objective in complex dynamics.
An important step towards this objective involves constructing combinatorial models for such families.
The most famous case when such a model is known is the \emph{quadratic family}, i.e., the family of quadratic polynomials $P_c(z)=z^2+c$.
The set of all parameters $c$ such that $P_c$ has connected Julia set (equivalently, the critical
$P_c$-orbit $0$, $P_c(0)=c$, $\dots$, is bounded) is called the \emph{Mandelbrot set} and is denoted by $\M_2$.
In his seminal preprint \cite{thu85}, William Thurston introduced a variety of (mostly combinatorial and geometric by nature) new tools and constructed a combinatorial geometric model $\M^c_2$ of $\M_2$.
There exists a monotone map from $\M_2$ onto $\M^c_2$.
A major open conjecture in complex dynamics, referred to as MLC, states that this map is a homeomorphism.

The set $\M^c_2$ has a very rich and fascinating structure.
It contains a countable and dense family of homeomorphic copies of itself
whose precise location can be deduced from the description of $\M^c_2$.
Thus, $\M^c_2$ is an example of a so-called \emph{fractal} set.
According to Adrien Douady, the process of constructing $\M^c_2$ can be described as ``pinching the closed unit disk $\cdisk$''
which is why $\M^c_2$ is often said to be the ``pinched disk model'' of $\M_2$.
``Pinching'' refers to collapsing a certain chord of $\cdisk$;
evidently, each act of pinching creates an increasingly complicated new quotient space of $\cdisk$.
Therefore it is natural to attempt to understand the structure of the ``pinched disk model'' by doing only
some of the pinchings and ignoring the other ones.
Somewhat loosely, one can say that in order to understand $\M^c_2$, one may try to partially ``unpinch'' $\M^c_2$ and thus obtain partial quotient spaces of $\cdisk$ as steps towards understanding $\M^c_2$.

This by itself provides a motivation for studying partially ``unpinched'' versions of $\M^c_2$.
There is however another no less important thought that motivated the authors.
Namely, while in the quadratic case the problem of constructing an adequate combinatorial
model of $\M_2$ was solved by Thurston in \cite{thu85}, producing a similar combinatorial model in the cubic case, let alone in the arbitrary degree $d$ case, presents a difficult problem that has not yet been solved.
Partial quotients of $\cdisk$ admit cubic analogs.
The latter may serve as simplified models of the cubic connectedness locus.

\begin{figure}[!htb]
    \centering
        \includegraphics[height=0.4\textheight]{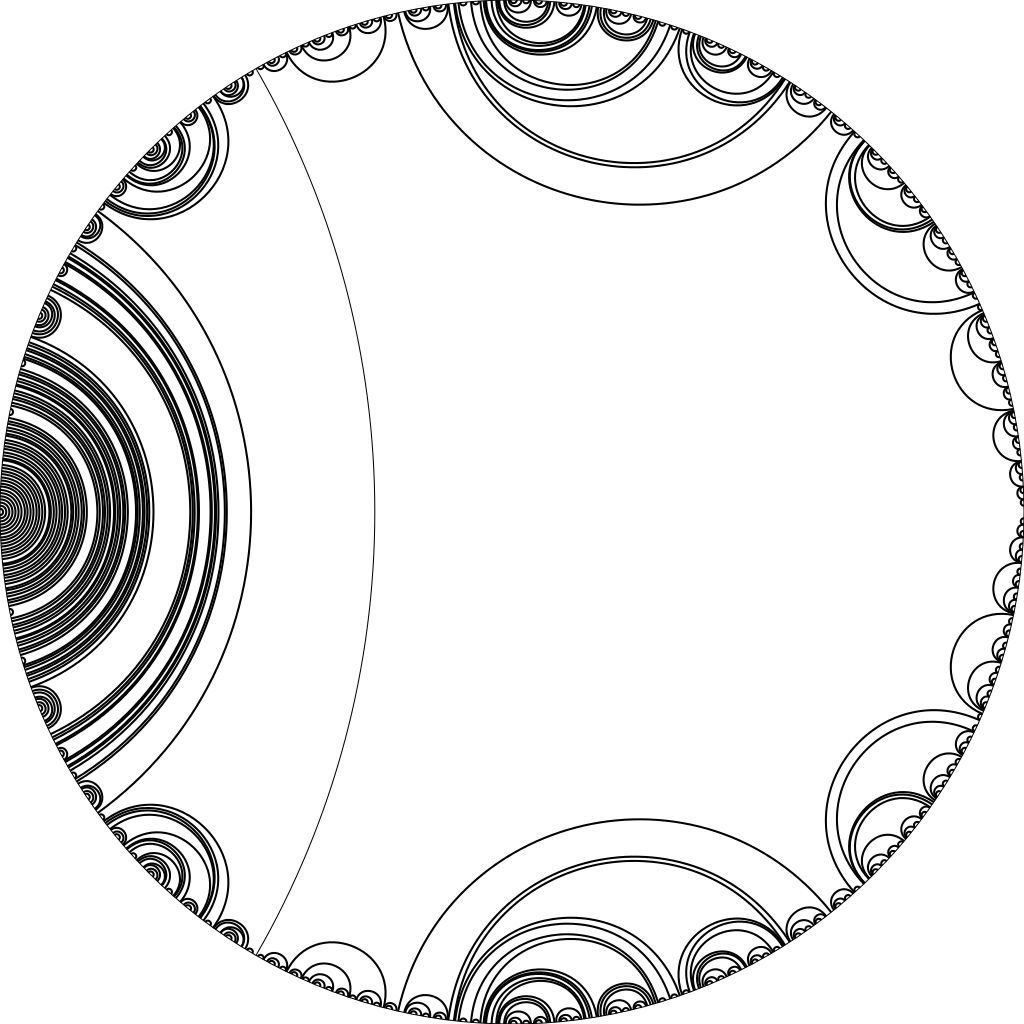}
        \caption{The geodesic lamination $\qml$}
        \label{fig:qml}
\end{figure}

We begin with an informal overview.
We assume basic knowledge of complex dynamics.
Later on, especially in the last two sections of the paper, we will also use well-known facts concerning the
structure of the combinatorial Mandelbrot set $\M^c_2$.

\subsection*{Laminational equivalence relations}
\emph{Laminational equivalence relations} are closed equivalence
relations $\sim$ on the unit circle $\uc$ in the complex plane $\C$
such that all classes are finite and the convex hulls of all classes
are pairwise disjoint; denote the corresponding quotient map of $\uc$
by $\psi_\sim$. Notice that $\psi_\sim$ can be canonically extended
over $\C$ by assuming that all classes of points of $\C$ outside
$\cdisk$ are these points themselves while inside $\cdisk$
any convex hull of a $\sim$-class is declared to be
one class of equivalence while the points inside infinite gaps of
$\lam_\sim$ are treated in the same way as points outside of $\cdisk$,
i.e. are declared to be degenerate classes. This extends $\sim$
and $\psi_\sim$ onto the entire complex plane. We are interested in
quotient spaces of $\cdisk$ and $\uc$ under the map $\psi_\sim$.

Each laminational equivalence relation $\sim$ gives rise to an
associated geometric object $\lam_\sim$ called a \emph{geodesic
lamination generated by $\sim$}. It is the set of all chords (called
\emph{leaves}) in the boundaries of the convex hulls (in the closed
unit disk) of all equivalence classes of $\sim$. The union of all leaves of
$\lam_\sim$ and $\uc$ is a continuum $\lam_\sim^+$ called the \emph{solid}
of $\lam_\sim$ (or of $\sim$). The closure in $\C$ of a non-empty
component of $\disk\sm \lam^+_\sim$ is called a \emph{gap} of $\lam_\sim$.
If $G$ is a gap or a leaf, call the set $G'=\uc\cap G$ the \emph{basis
of $G$}. A gap is said to be \emph{finite $($infinite, countable,
uncountable$)$} if its basis is finite (infinite, countable,
uncountable); countable gaps are only possible for
generally defined geodesic laminations (see below) but not for the just
defined geodesic lamination generated by $\sim$. For $\lam_\sim$ infinite
gaps are associated with bounded complementary domains of $\uc/\sim$ by
the corresponding quotient map.

Sometimes we relax the assumption of finiteness of
classes of equivalence and talk of laminational
equivalence relations \emph{possibly with infinite classes}
\cite{bopt17}. All related notions can be defined for either
laminational equivalence relations or for laminational equivalence
relations possibly with infinite classes, thus in what follows we only
define them for laminational equivalence relations.
In particular, this concerns the associated quotient maps and quotient
spaces.
In case when we deal with a laminational equivalence relation
$\sim$, all the infinite gaps of $\sim$ give rise to bounded
complementary domains of the corresponding quotient space.

Assume now that $\lam_\sim$ is \emph{perfect} (i.e., there are no
isolated leaves in $\lam_\sim$). Then, no two gaps of $\lam_\sim$
intersect. There are various laminational equivalence relations
possibly with infinite classes that generate $\lam_\sim$. This is
because some infinite gaps of $\lam_\sim$ can be declared convex hulls
of equivalence classes (recall that $\lam_\sim$ is perfect!). Now,
consider the laminational equivalence relation $\sim^d$ in which
\emph{all} infinite gaps of $\lam_\sim$ are declared convex hulls of
classes (equivalently, for each infinite gap $G$ of $\lam_\sim$ its
basis $G'$ is declared one class). It is easy to see that
$\uc/\sim^d=\cdisk/\sim^d$ is a \emph{dendrite} (a locally connected
continuum without subsets homeomorphic to a circle). This dendrite can
be obtained from $\cdisk/\sim$ by collapsing the closures of all bounded complementary
domains of $\uc/\sim$ (recall that they are associated to infinite gaps
of $\lam_\sim$) to points. We call $\sim^d$ and $\cdisk/\sim^d$ the
\emph{dendritic version} of $\sim$ and $\cdisk/\sim$.

A laminational equivalence relation is ($\si_d$-)\emph{invariant} if it is ``respected'' by the map $\si_d(z)=z^d:\uc\to\uc$; a formal definition will be given later.
The map $\si_d$ induces a \emph{topological polynomial}
$f_\sim:\uc/\hspace{-4pt}\sim\,\to \uc/\hspace{-4pt}\sim$ from the \emph{topological Julia set} $J_\sim=\uc/\hspace{-4pt}\sim$ to itself.
If, for a polynomial $P$, its Julia set $J(P)$ is locally connected, then $J(P)$ can be identified with the topological Julia set $\uc/\sim$ of some laminational equivalence relation $\sim$ such that  $\psi_\sim$ identifies with the extension to the unit circle of the normalized Riemann map of the basin of infinity of $P$.
Thus, in this case, $P|_{J(P)}$ is topologically conjugate to (and even identifies with) $f_\sim$.

From now on, ``invariant'' means ``invariant under the map $\si_d$''.
In particular, given a chord $\ell=\ol{ab}\in\lam_\sim$, with endpoints $a, b\in\uc$, the chord
$\ol{\si_d(a)\si_d(b)}$ is also a (possibly degenerate) chord of $\lam_\sim$, and we write $\si_d(\ell)=\ol{\si_d(a)\si_d(b)}$.
Not only should the collection $\lam_\sim$ of chords be invariant but also boundaries of complementary to $\lam_\sim$
domains of $\cdisk$ should map forward in the orientation preserving fashion (except that some edges may collapse to points).

\subsection*{Invariant geodesic laminations}
Although $\si_d$ is only defined on the unit circle $\uc$, it makes sense to talk about the images $\si_d(X)$ for subsets $X$ of the closed unit disk
  obtained as convex hulls of $X\cap\uc$; by definition, we set $\si_d(X)$ to be the convex hull of $\si_d(X\cap\uc)$.
Thurston in his famous paper \cite{thu85} introduced and studied geometric objects similar to $\lam_\sim$ but more general, and called them \emph{invariant geodesic laminations}.
Thus, $\si_d$-invariant geodesic laminations $\lam$ are closed collections of chords in $\cdisk$ (\emph{leaves})
with special properties relating the leaves and the map $\si_d$. It will be convenient to consider
also unions of all leaves of $\lam$ together with all points of $\uc$; such unions
are denoted $\lam^+$ and are called \emph{solids} of laminations.
For any laminational equivalence relation $\sim$, the set $\lam_\sim$ is a $\si_d$-invariant geodesic lamination.
Such geodesic laminations are called $q$-laminations ($q$ from ``e\textbf{q}uivalence'').

An important reason for defining geodesic laminations is to provide ``visualization'' of laminational
equivalence relations and use the former to ``topologize'' the latter.
This is done by defining \emph{Hausdorff metric}  on geodesic laminations.
To this end, consider the space of all chords (including degenerate ones) in the unit disk.
Equivalently, we can consider the space of all unordered pairs of points of $\uc$ (so that a pair $(\al, \be)$ is identified with the pair $(\be, \al)$).
It follows that the space in question is a ``half-torus'' $E$.
If a topological 2-torus is represented as a quotient space of the square $0\le x,y\le 1$ by the equivalence relation identifying $(x,0)$ with $(x,1)$ and $(0,y)$ with $(1,y)$ for all $x$, $y\in [0,1]$, then $E$ is given by $x\le y$ with the additional identification of $(0,x)$ with $(x,1)$.
Topologically, this is a M\"obius strip.
Every geodesic lamination $\lam$ can be viewed as a closed subset of $E$ (each leaf of $\lam$ is a point of $E$).
Define the Hausdorff distance between two geodesic laminations $\lam_1$, $\lam_2$ by viewing $\lam_1$ and $\lam_2$ as subsets of $E$.
When speaking of geodesic laminations, we will always consider the topology on them defined by the Hausdorff metric.
We will speak of limits of laminations only in this sense.

If $d=2$, then the corresponding geodesic laminations, laminational equivalence relations, topological polynomials and Julia sets are said to be \emph{quadratic}.
However, for clarity of terminology, in this paper we address the corresponding geodesic laminations and laminational equivalence relations as \emph{$\si_2$-invariant}.
Thurston \cite{thu85} introduced and studied properties of individual $\si_d$-invariant laminations while also
obtaining striking results on the entire family of $\si_2$-invariant geodesic laminations and parameterizing them with the help of his famous \emph{quadratic minor lamination} $\qml$.

For every $\si_2$-invariant geodesic lamination $\lam$, Thurston defines its \emph{major} $M$ as a longest leaf of $\lam$.
It is easy to see that either $\lam$ has a unique major, which is then a diameter of $\cdisk$, or $\lam$ has two distinct majors with equal $\si_2$-images.
Then Thurston defines $m=\si_2(M)$, calls it the \emph{minor} of $\lam$, and shows that \emph{the family of the minors of all $\si_2$-invariant geodesic laminations is a geodesic lamination itself!}
This geodesic lamination is the \emph{quadratic minor lamination} $\qml$ that can be generated by an equivalence relation $\sim_\qml$.
Each class of $\sim_\qml$ is associated with a unique $\si_2$-invariant laminational equivalence relation and the
corresponding topological polynomial.
The quotient space $\uc/\hspace{-4pt}\sim_\qml$ parameterizes the space of all $\si_2$-invariant
laminational equivalence relations (equivalently, the space of all topological polynomials of degree two).
This yields a set $\M^c_2$ in the plane, defined up to a homeomorphism, whose boundary $\bd(\M^c_2)$ is homeomorphic to $\uc/\sim_\qml$.

Geodesic laminations were introduced not only to study dynamics of individual complex polynomials but also to
construct combinatorial (laminational) models for spaces of complex polynomials.
This is done with the help of the following approach.
Firstly, with each complex polynomial from the space under consideration, associate its laminational equivalence relation and the corresponding topological
polynomial.
Observe that, while this is straightforward in the locally connected case and in some other cases (see below), in case of more complicated topology of the Julia sets a special construction may be required.
Secondly, describe (parameterize) the space of topological polynomials associated with all complex polynomials from the chosen space, and use the thus obtained space of topological polynomials as the model for the given space of complex polynomials.

In order to implement this plan, one has to overcome several difficulties.
First, the association of a laminational equivalence relation with a complex polynomial may be hard.
This is related to the fact that the Julia set $J(P)$ of a complex polynomial $P$ may have complicated topology.
Still, in many cases $P$ gives rise to a lamination, even if $J(P)$ is not locally connected.
For example, this happens in the case when $J(P)$ is connected and has no periodic non-repelling points.
In this case, the association is due to Jan Kiwi \cite{kiwi97}.
The corresponding topological Julia sets are dendrites, and the corresponding complex polynomials are called \emph{dendritic}.
Notice that the actual Julia set of a dendritic polynomial does not have to be a dendrite.

Known facts from polynomial dynamics imply that dendritic polynomials are dense in the boundary of the Mandelbrot set.
Thus, one can extend the correspondence between complex quadratic dendritic polynomials and their laminational equivalence relations onto $\bd(\M_2)$.
We mean the extension by continuity, or the closure of the correspondence.
This closure fails to be a map; one polynomial may correspond to several laminations.

Let us now discuss details of the above outlined plan focusing on the case of quadratic polynomials.
A difficulty here is that geodesic laminations associated to laminational equivalence relations do not form a closed subspace.
In \cite{bopt16}, we resolve this issue by taking the closure $\ol{\L^q_2}$ of the space $\L^q_2$ of all $q$-laminations.
Clearly, $\ol{\L^q_2}$ includes all $q$-laminations as well as limits of non-constant sequences of $q$-laminations.

\subsection*{Limit laminations}
We proceed in \cite{bopt16} by following Thurston \cite{thu85} and defining \emph{minors $m(\lam)$} of geodesic laminations from $\ol{\L^q_2}$.
Two geodesic laminations $\lam_1$, $\lam_2\in \ol{\L^q_2}$ are identified if $m(\lam_1)\cap m(\lam_2)\ne \0$.
This yields the associated quotient space of $\ol{\L^q_2}$ which is proven in \cite{bopt16} to be homeomorphic to $\M^c_2$.
However, we want to argue that there is another similar quotient space of a certain subspace of $\ol{\L^q_2}$ that is, from the point of view of geodesic laminations, in some sense even more natural than $\M^c_2$ (viewed as a quotient space of $\ol{\L^q_2}$).
Indeed, the space $\ol{\L^q_2}$ may contain isolated geodesic laminations.
As we visualize the space $\ol{\L^q_2}$, isolated laminations can be placed anywhere in the picture.
In other words, their location is a matter of convenience and does not reflect the true structure of the space.
Therefore, it is of interest to consider the space of all non-isolated geodesic laminations from $\ol{\L^q_2}$ with the same identification as above (two are identified if their minors are non-disjoint), and to describe the corresponding quotient space.

\emph{Limit laminations} are defined as non-isolated points in $\ol{\L^q_2}$.
It is proven in \cite{bopt16} that a $\si_2$-invariant limit lamination $\lam$ must contain a critical quadrilateral or a critical leaf.
However, $\lam$ having a critical quadrilateral or leaf is only a necessary condition.
In this paper, we give sufficient conditions and thus describe the entire space of limit laminations, which will be denoted by $\L^l_2$ (``l'' from ``limit'').
It turns out that the sufficient condition on $\lam$ is the existence of a $q$-lamination $\lam_\sim$ satisfying property (0) below and one of properties (1)--(3).

(0) The critical set of $\lam$ is included into the critical set of $\lam_\sim$.

(1) The lamination $\lam_\sim$ is dendritic with a critical set $X$; in this case,
$\lam$ is obtained from $\lam_\sim$ by inserting a critical leaf/quadrilateral $Q$ into $X$ and inserting pullbacks of $Q$ into pullbacks of $X$.
Such geodesic lamination is a limit lamination.
A \emph{critical leaf} is by definition a diameter of the unit circle.
A \emph{critical quadrilateral} is by definition the convex hull of two critical leaves.

(2) The lamination $\lam_\sim$ has a periodic Siegel gap, and $\lam=\lam_\sim$.

(3) There is a periodic Fatou gap $U$ of $\lam_\sim$ of period $n$ that maps forward two-to-one under $\si_2$.
It is well known \cite{thu85} that then there exists a unique non-degenerate periodic edge $M$ of $U$ of exact $\si_2$-period $n$.
The leaf $M$ is a \emph{major} of $\sim$ while $m=\si_2(M)$ is the \emph{minor} of $\sim$ (and of $\lam_\sim$).
There are two possible subcases.

(a) The leaves $M$, $\si_2(M)$, $\dots$, $\si_2^{n-1}(M)$ are pairwise disjoint.
The critical set of $\lam$ is either a critical leaf sharing an endpoint with $M$ or the critical quadrilateral having $M$ as an edge.
All these laminations $\lam$ can be described explicitly, and all of them are identified.

(b) The leaves $M$, $\si_2(M)$, $\dots$, $\si_2^{n-1}(M)$ are not pairwise disjoint.
The lamination $\lam$ has a critical diameter sharing an endpoint with $M$.

Thus, the minors of laminations $\lam$ of type b) are endpoints of $m=\si_2(M)$.
These laminations $\lam$ break down into two classes corresponding to the endpoints of $m$.
These are distinct classes because no limit geodesic lamination has a critical quadrilateral with edge $M$.

Laminations of type b) are exactly laminations that correspond to points of $\M^c_2$ at which the closures of two distinct hyperbolic domains in $\M^c_2$ meet. According to the description in b), corresponding leaves in $\qml$ must be erased and replaced by their endpoints.
Mimicking Thurston, we can call this process \emph{cleaning of $\qml$}.
This is the only modification of $\qml$ necessary to transform $\sim_\qml$ into a laminational equivalence relation parameterizing the above described
quotient space of the space of all $\si_2$-invariant limit laminations.

By the known properties of $\qml$, this cleaning produces a new geodesic lamination $\qml^l$ with no isolated leaves (and hence perfect).
The corresponding quotient space will be denoted by $\M^l_2$.
Informally, one can say that while the ``pinched disk model'' of $\M^c_2$ is obtained by pinching $\cdisk$ at various places prescribed by $\qml$, we can then obtain $\M^l_2$ by ``unpinching'' $\M^c_2$ at places prescribed by $\qml\sm\qml^l$.
Alternatively, one can say that $\M^l_2$ is obtained if certain parts of $\M^c_2$ are replaced by topological disks.
Two interior components of $\M^c_2$ are said to be \emph{associated} if they are connected by a chain of complementary components touching at boundary points.
The closures of maximal unions of pairwise associated components are to be replaced with topological disks.

An alternative description of the space $\L^l_2$ is that the latter
consists of limits of sequences of \emph{dendritic} geodesic
laminations. It is easy to see that now we do not have to insist that
the sequences be non-constant. Observe also that there are other,
alternative ways to describe the laminations from $\L^l_2$ by
describing their properties intrinsically and without referring to the
limits of geodesic laminations. Indeed, all laminations from $\L^l_2$
have no isolated majors \emph{unless} the critical set is just a
critical leaf (as is the case of a Siegel gap). This can be turned into
a characterization of $\L^l_2$.

After ``unpinching'', the Main Cardioid $\ca$ turns into a bigger gap denoted here by $\ca^l$.
In fact, we can think of ``unpinching'' as follows.
On the first step we erase all non-degenerate edges of the Main Cardioid itself.
On the second step we erase non-degenerate edges in the copies of the Main Cardioid that used to be attached
to the Main Cardioid. If we go on with this, we will in the end obtain $\ca^l$. Observe that some
degenerate edges of $\ca^l$ are obtained on finite steps in the process. These are endpoints of edge
erased after a finite number of steps
or Siegel points on the boundary of some copy of the Main Cardioid.

Evidently, the only remaining edges of $\ca^l$ that are not obtained
after a finite number of steps are infinitely-renormalizable and are
limits of sequences of non-degenerate edges of deeper and deeper copies
of $\ca$ associated with $\ca$. Some of these edges are non-degenerate,
others are degenerate. Let us characterize those minors of infinitely
renormalizable laminations that are non-degenerate (we provide here
only a sketch of the proof). To this end we need a well-known
description of canonical copies of $\M^c_2$ contained in $\M^c_2$
(sometimes these are informally called ``baby Mandelbrot sets'').

Indeed, these copies are obtained by first choosing specific periodic (of period,
say, $k$) infinite gaps $V$ that are $\si_2$-images of quadratic Fatou
gaps. The standard monotone collapse of edges of $V$ semiconjugates
$\si_2^k|_{\bd(V)}$ with $\si_2$ (see details in
Section~\ref{s:non-renorm}). Using this correspondence, we lift $\qml$
to $V$. It turns out that the lift in fact coincides with the
restriction of $\qml$ onto $V$. If we take the associated quotient
space of $V$, then we will obtain a copy of $\M^c_2$ inside $\M^c_2$.
In other words, one chooses a special gap $V\subset\cdisk$ and then
considers a ``pinched'' version of $V$ using $\qml$ transferred to $V$
and inducing the corresponding lamination in $V$. This is what makes
$\M^c_2$ a \emph{fractal}, i.e., a set that contains infinitely many
homeomorphic copies of itself.

The part of $\M^c_2$ important for us is the real line (i.e., one can
draw the real line and consider its segment contained in $\M^c_2$).
There are two equivalent ways of describing this part of $\M^c_2$.
First, observe that all minors that intersect the real line are
\emph{vertical} and vice versa so that the fact that a minor is
vertical is equivalent to it belonging to the real line in $\M^c_2$. To
define the second way of describing the real line inside $\M^c_2$ we
need the notion of the \emph{growing tree of $f_\sim$} \cite{bl02}. It
is defined as follows. In the topological Julia set $J_\sim$ we connect
the two points associated with angles $0$ and $1/2$ (the unique fixed
angle for $\si_2$ and its preimage). Denote this connecting arc $I$.
Then consider the union of all images of $I$ under $f_\sim$ by
$T_\infty$ and call it the \emph{growing tree of $f_\sim$}. It is easy
to see (\cite{bl02}) that the set $T_\infty$ is an interval if and only
if the minor of $\lam_\sim$ is vertical.

Now, if $\lam_\sim$ is infinitely renormalizable then $J_\sim$ is a
dendrite. Suppose that $\lam_\sim$ is associated with a real polynomial
(equivalently, the two majors of $\lam_\sim$ are vertical). Clearly,
these majors cannot coincide (and be critical leaves) as the only
vertical diameter is $\ol{\frac14\dfrac34}$ which does not correspond
to an infinitely renormalizable lamination. Thus, the critical set of
$\lam_\sim$ is a critical quadrilateral with vertical edges and the
minor of $\lam_\sim$ is non-degenerate. It follows, that if we choose a
canonical copy of $\M^c_2$ inside $\M^c_2$ and then choose similar
place inside this copy, the corresponding minor will be non-degenerate.
For brevity let us call minors of this kind \emph{renormalized real}.
Thus, for every renormalized real minor of an infinitely renormalizable
lamination $\lam_\sim$ is non-degenerate.

To prove the converse, consider $J_\sim$. If the critical set of
$\lam_\sim$ is a non-degenerate critical quadrilateral $Q$, it follows
that $f_\sim(c)$ is a cutpoint of $J_\sim$ (here $c$ is the critical
point of $f_\sim$). Hence by \cite{bl02} there exists a finite
$f_\sim$-invariant tree $T$ that contains $c$. If we consider
sufficiently high renormalization of $f_\sim$ it will have to be
defined on the interval because there are only finitely many
branchpoints in $T$. This implies that for this renormalization its
growing tree is an interval. By the above this implies that the
corresponding minor is real. To summarize, the minor of an infinitely
renormalizable lamination $\lam_\sim$ is non-degenerate if and only if
it is renormalized real.

The geodesic lamination $\qml^l$ is a visual counterpart of a certain
equivalence relation on the circle. As $\qml^l$ is perfect, there are
two main alternative equivalence relations, $\sim^l$ and $\sim^{ld}$,
that generate the same geodesic lamination $\qml^l$. Here $\sim^l$ is a
laminational equivalence relation (so that all its classes are finite),
and since $\qml^l$ is perfect, $\sim^l$ is well-defined. By
construction, $\qml^l$ has a variety of pairwise disjoint infinite
gaps. The interpretation of $\qml^l$ as the geodesic lamination
generated by a laminational equivalence relation is based on the
assumption that all classes are finite so that in the quotient,
infinite gaps of $\qml^l$ give rise to bounded complementary domains of
$\M^l_2$. As explained above, in a alternative dendritic version of
$\qml^l$ we declare \emph{every} gap of $\qml^l$ to be the convex hull
of an equivalence class. This results into a well-defined laminational
equivalence relation with possibly infinite classes $\sim^{ld}$. Hence,
one can define yet another quotient space $\uc/\sim^{ld}=\M^{ld}_2$, a
dendritic version of $\M^l_2$, by collapsing the closures of all
bounded complementary domains to $\M^l_2$.

This purely topological construction in fact defines a certain quotient space of $\L^q_2$ (or $\ol{\L^q_2}$, or $\L^l_2$).
Namely, recall that the \emph{perfect part} of a geodesic lamination $\lam$ is obtained  by taking the maximal perfect subset of $\lam$.
In particular, all isolated leaves of $\lam$ must be erased as we extract the perfect part of $\lam$.
Now, say that two geodesic laminations are \emph{countably equivalent} if they have the same perfect parts
(equivalently, if the symmetric difference between them is countable).

For instance, all laminations from the Main Cardioid are countably equivalent.
Their common perfect part is the unit circle.
Other countable laminations 
also have $\uc$ as their perfect part.
All countable laminations are countably equivalent. However countable equivalence relation defined above
is not closed. If we close it and consider the resulting class of equivalence of all countable laminations,
we will get the class coinciding with the boundary of $\ca^l$ intersected with
the unit circle (recall that when we talk about classes of equivalence we consider points of the unit circle).
Slightly abusing the language, let us call the just defined closed equivalence relation \emph{countable equivalence relation}
and say that two geodesic laminations are \emph{countably equivalent} if they belong
to the same class in the sense of countable equivalence relation.
Then $\ca^l$ is the convex hull of the class of countable equivalence relation corresponding to the Main Cardioid.
The quotient space of the unit disk/circle under the countable equivalence relation is
$\uc/\sim^{ld}=\M^{ld}_2$ defined above.

\begin{figure}[!htb]
    \centering
    \begin{minipage}{0.5\textwidth}
        \centering
        \includegraphics[height=0.2\textheight]{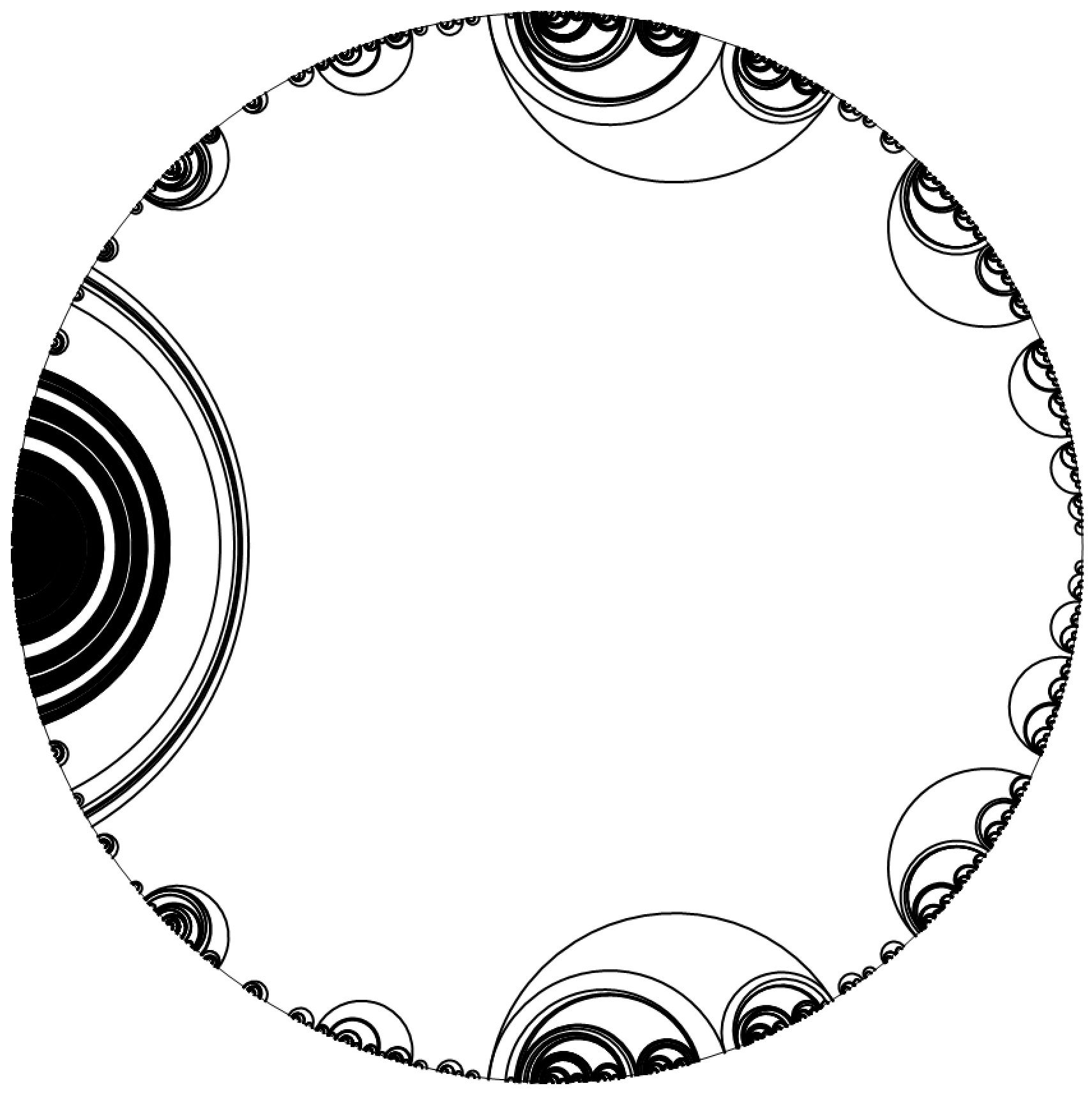}
        \caption{The lamination $\qml^l$}
        \label{fig:mset}
    \end{minipage}%
    \begin{minipage}{0.5\textwidth}
        \centering
        \includegraphics[height=0.2\textheight]{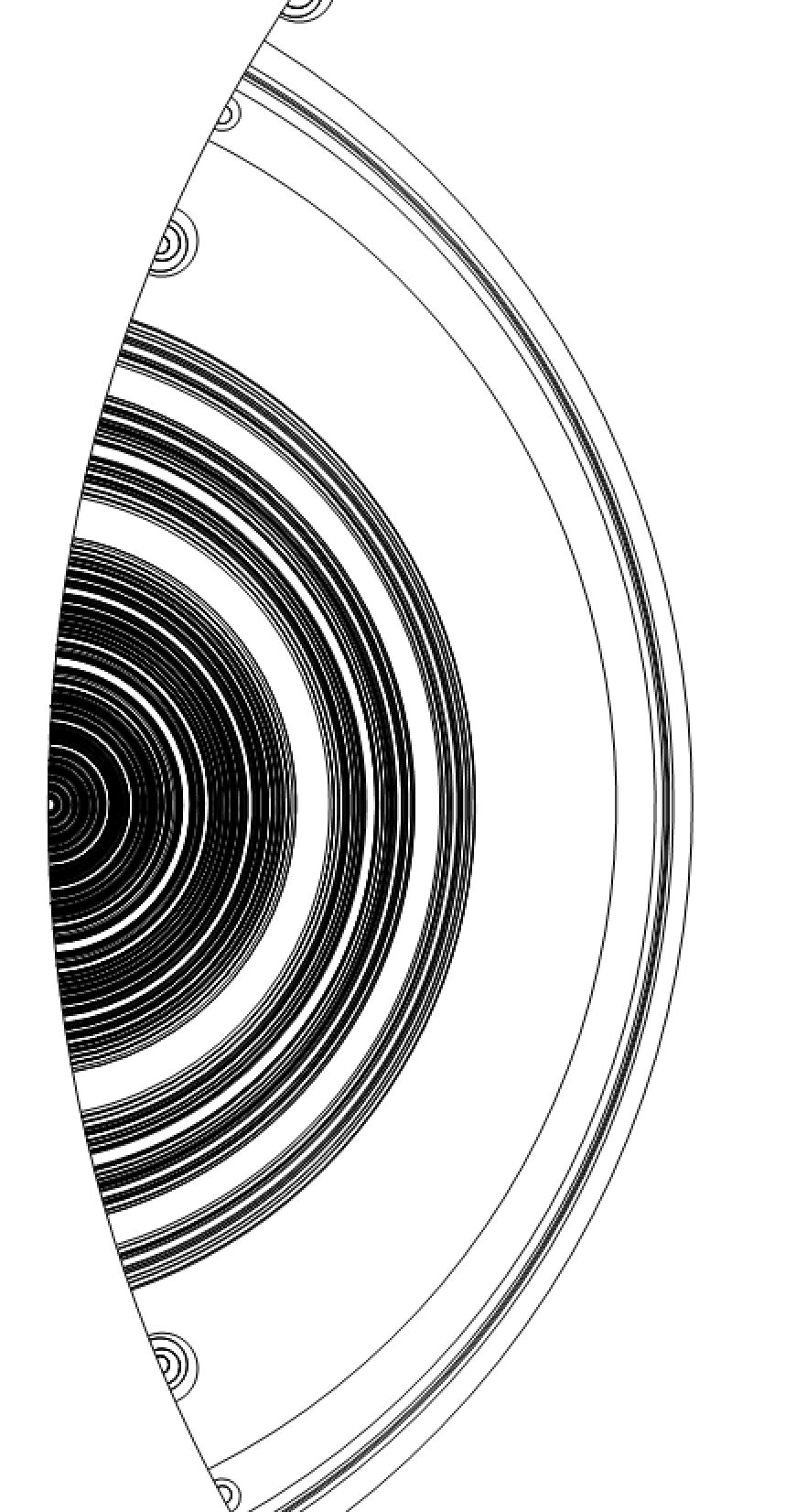}
        \caption{A zoom-in of $\qml^l$}
        \label{fig:qml^l}
    \end{minipage}
\end{figure}

\begin{figure}[!htb]
    \centering
        \includegraphics[height=0.2\textheight]{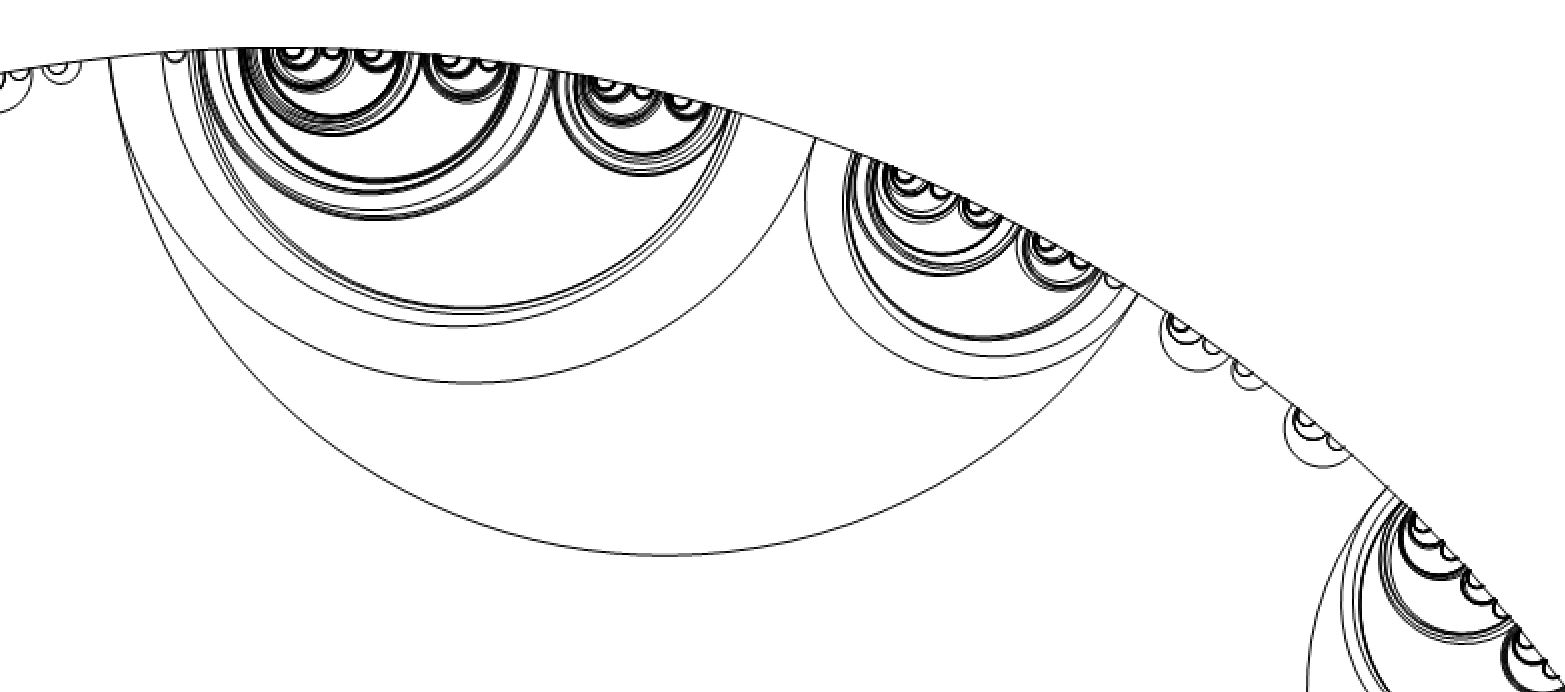}
        \caption{Another zoom-in of $\qml^l$}
        \label{fig:qml^l2}
\end{figure}

\subsection*{Parameter space of non-renormalizable laminations}
In the last section of the paper, Section~\ref{s:non-renorm}, we consider another way to modify $\M^c_2$.
The aim, as in the cases of $\M^l_2$ and $\M^{ld}_2$ above, is to uncover the structure of $\M^c_2$ by
replacing more complicated parts of $\M^c_2$ by their simplified ``unpinched'' versions.
Recall that ``unpinching'' works as follows.
For certain leaves $\ell$ of $\M^c_2$, we allow laminations from $\ol{\L^q_2}$ with critical leaves sharing endpoints with $\ell$ and not allow laminations with the critical quadrilateral supported by $\ell$.
The name of the game is which critical quadrilaterals are allowed and which are not.

We now want to remove from $\qml$ all non-degenerate leaves
corresponding to renormalizable laminations. Renormalizable
$q$-laminations can be described as laminations $\lam_\sim$ such that
$J_\sim$ admits \emph{non-trivial} periodic subcontinua (under the
action of $f_\sim$). Equivalently, $\lam_\sim$ is renormalizable if
$\lam^+_\sim$ has a \emph{non-trivial} $\si_2$-periodic closed
connected subset. Here ``non-trivial'' means non-degenerate and not
equal to the entire Julia set (or, in the second case, not projected to
a trivial continuum under the quotient map).

Now, suppose that $\lam_1$ is a renormalization of $\lam_2$. It means
the following. First, there exists a lamination $\hlam_2$ with an
$n$-periodic critical Fatou gap $U$. The lamination $\lam_2$ contains
$\hlam_2$ (i.e., $\hlam_2\subset \lam_2$) so that all the leaves of
$\lam_2\setminus \hlam_2$ are contained in gaps of $\hlam_2$ that
belong to the grand orbit of $U$. Now, restrict $\lam_2$ onto $U$ and,
similar to the above, collapse all edges of $U$ to points. This
semiconjugates $\si_2^n|_{\bd(U)}$ and $\si_2$ (intuitively, one can say that
this ``magnifies $U$ to the unit circle) and transforms $\lam_2|_U$ to
$\lam_1$. In this setting $\lam_2$ is said to be a \emph{tuning} of
$\hlam_2$.

In what follows it is easier for us to work with tunings of geodesic
laminations rather than with their renormalizations (even though these
two approaches are clearly equivalent). Thus, given a geodesic
lamination $\lam_1$ we may be able to find a lamination $\lam_2$
with $\lam_1\supset \lam_2$ such
that $\lam_1$ is a tuning of $\lam_2$. Now, $\lam_2$ can itself be a
tuning of a lamination $\lam_3\subset \lam_2$, etc. If we go on
with this, on the $n$-th step we will obtain a sequence of geodesic
laminations $\lam_1\supset \lam_2\supset \lam_3\supset \dots\supset
\lam_{n+1}$ so that $\lam_i$ is a tuning of $\lam_{i+1}\subset \lam_i$ as long as $i\le
n$. In the obvious sense $\lam_2, \lam_3, \dots, \lam_{n+1}$ are \emph{ancestors} of $\lam_1$
while $\lam_1$ is an \emph{offspring} of $\lam_2, \dots, \lam_{n+1}$. Since
we want to ``unpinch'' all minors associated with renormalizable
laminations, we need to move along this ``genealogical tree'' as far
back as possible in order to find the \emph{oldest ancestor} of a given
geodesic lamination $\lam=\lam_1$.

There is an important detail that should not be overlooked here.
Namely, every geodesic lamination is a tuning of the \emph{empty}
lamination (the one without non-degenerate leaves). This trivializes
the issue of looking for the oldest ancestor as the empty lamination is
the oldest ancestor for all laminations. Thus, we need to decide in
what cases the empty lamination can(not) be considered as the oldest
ancestor of a lamination. To this end, consider a hyperbolic lamination
$\lam$ with critical Fatou gap $U$ such that $\lam$ is \emph{not} a
tuning of any non-empty lamination.
(A hyperbolic lamination is a lamination corresponding to a hyperbolic polynomial).
Making one step back from $\lam$ to
the empty lamination should not be allowed if, as a result, a
substantial piece of information is lost. Thus, if $\lam$ is
uncountable, we consider $\lam$ as the \emph{oldest ancestor lamination}
and call $\lam$ \emph{non-trivial}. However, if $\lam$ is countable
(recall that $\lam$ is assumed to not be a tuning of any non-empty
lamination so that its countability implies that $\lam$ belongs to the
Main Cardioid) we consider the empty lamination as the corresponding
\emph{oldest ancestor lamination}. This defines the family of all
oldest ancestor laminations. The empty
lamination is called the \emph{trivial} oldest ancestor lamination.

Consider now a new family of laminations. We keep the
non-re\-nor\-ma\-li\-zable laminations in it. Otherwise for any
non-trivial oldest ancestor lamination with critical Fatou gap $U$ we
consider all possible laminations with critical leaves in $U$ or
critical quadrilaterals based upon edges of $U$. We will still
characterize (``tag'') our laminations with their minors. Thus,
postcritical gaps $V$ of non-trivial oldest ancestor laminations
described above and pinched under the equivalence relation $\sim_\qml$
in the process of creation of $\M^c_2$ are now completely
``unpinched''. Thus, gaps $V$ must be present in the parametric
lamination $\qml^{nr}$ that we are now constructing. Notice that the
resulting lamination $\qml^{nr}$ does not have any isolated leaves. As usual,
denote by $\sim^{nr}$ the laminational equivalence relation generating $\qml^{nr}$.

Now consider all \emph{countable} laminations $\lam$ that are not tunings of
non-empty laminations. Recall that we associate with them the empty
lamination as their oldest ancestor. As above, we insert all possible
critical leaves into them and all possible critical
quadrilaterals into them, with one exception. Namely, if $U$ is a
critical gap of $\lam$ of period $n$ then we do \emph{not} insert a
critical quadrilateral in $U$ supported by the edge of $U$ of period $n$.
This corresponds to the fact that all countable laminations $\lam$ that
are not tunings of non-empty laminations have the same empty oldest
ancestor and therefore should be associated with the same gap
$\ca^{nr}$ in the parameter space.

In $\qml^{nr}$ the gap $\ca^{nr}$
replaces the Main Cardioid. It plays a special role in $\qml^{nr}$. 
To describe its
edges and vertices, take all edges and vertices of all postcritical
Fatou gaps of laminations from the Main Cardioid --- except
for the edges of the same period as the Fatou gap in question (i.e.,
except for the edges that are edges of the Main Cardioid itself). Add
to this list all vertices of the Main Cardioid associated with Siegel
laminations. This collection of edges and vertices forms the boundary
of $\ca^{nr}$. Basically this means that to create $\ca^{nr}$ we erase
all minors of $\qml$ contained in postcritical Fatou gaps associated
with laminations from the Main Cardioid, and also all edges of the Main
Cardioid.

As with limit laminations and their space, alternatively to $\sim^{nr}$
in the situation of Section~\ref{s:non-renorm} we can also identify all
points of $\uc$ on the boundaries of infinite gaps of $\qml^{nr}$. This
creates a laminational equivalence relation with possibly infinite
classes $\sim^{nrd}$ called the \emph{dendritic version} of
$\sim^{nrd}$. In terms of quotient spaces this will lead to the
collapse of all bounded complementary domains of $\M^{nr}_2$ to points
yielding a new, dendritic, parameter space denoted by $\M^{nrd}_2$.

\section{Preliminaries}\label{s:prelim}

A part of this section is devoted to geodesic laminations, a major tool in studying the dynamics
of individual complex polynomials as well as in modeling certain families of complex polynomials.
Let $a$, $b\in \uc$.
By $[a, b]$, $(a, b)$, etc., we mean the closed, open, etc., \emph{positively oriented} circle arcs from $a$ to $b$, and
by $|I|$ the normalized length of an arc $I$ in $\uc$  (a normalization is made so that the length of $\uc$ is $1$).

\subsection{Laminational equivalence relations}\label{ss:lam}

Denote by $\hc$ the Riemann sphere.
For a compactum $X\subset\C$, let $\iU(X)$ be the component of $\hc\sm X$ containing infinity.
If $X$ is connected, there exists a Riemann mapping $\Psi_X:\hc\sm\ol\bbd\to \iU(X)$;
we always normalize it so that $\Psi_X(\infty)=\infty$, and $\Psi'_X(z)$ tends to a positive real limit as $z\to\infty$.

Consider a monic polynomial $P$ of degree $d\ge 2$, i.e., a polynomial of the form $P(z)=z^d+$ lower order terms.
Consider the Julia set $J_P$ of $P$ and the filled-in Julia set $K_P$ of $P$.
Extend the map $z\mapsto z^d$ to a map $\ta_d$ on $\hc$.
If $J_P$ is connected, then $\Psi_{J_P}=\Psi:\hc\sm\ol\bbd\to \iU(K_P)$ is such that $\Psi\circ
\ta_d=P\circ \Psi$ on the complement of the closed unit disk \cite{hubbdoua85, mil00}.
If $J_P$ is locally connected, then $\Psi$ extends to a continuous function
$$
\ol{\Psi}: {\hc\sm\bbd}\to
\ol{\hc\setminus K_P},
$$
and $\ol{\Psi} \circ\,\ta_d=P\circ\ol{\Psi}$ on the complement of the open unit disk.
Thus, we obtain a continuous surjection $\ol\Psi\colon\bd(\bbd)\to J_P$ (the \emph{Carath\'eodory loop}).
Throughout the paper, $\bd(X)$ denotes the boundary of a subset $X$ of a topological space.
Identify $\uc=\bd(\bbd)$ with $\mathbb{R}/\mathbb{Z}$.
Set $\psi=\ol{\Psi}|_{\uc}$.
We will write $\si_d$ for the restriction of $\ta_d$ to $\uc$.

Define an equivalence relation $\sim_P$ on $\uc$ by $x \sim_P y$ if and
only if $\psi(x)=\psi(y)$, and call it the ($\si_d$-invariant) {\em
laminational equivalence relation of $P$}; since $\Psi$ defined above
conjugates $\ta_d$ and $P$, the map $\psi$ semiconjugates $\si_d$ and
$P|_{J(P)}$, which implies that $\sim_P$ is invariant. Equivalence
classes of $\sim_P$ have pairwise disjoint convex hulls. The
\emph{topological Julia set} $\uc/\sim_P=J_{\sim_P}$ is homeomorphic to
$J_P$, and the \emph{topological polynomial} $f_{\sim_P}:J_{\sim_P}\to
J_{\sim_P}$, induced by $\si_d$, is topologically conjugate to
$P|_{J_P}$.


An equivalence relation $\sim$ on the unit circle, with similar
properties to those of $\sim_P$ above,  can be introduced abstractly
without any reference to a complex polynomial.

\begin{dfn}[Laminational equivalence relations]\label{d:lam}

An equivalence relation $\sim$ on the unit circle $\uc$ is said to be
\emph{laminational} if:

\noindent (E1) the graph of $\sim$ is a closed subset in $\uc \times
\uc$;

\noindent (E2) convex hulls of distinct equivalence classes are
disjoint;


\noindent (E3) each equivalence class of $\sim$ is finite. 
\end{dfn}

For a closed set $A\subset \uc$ we denote its convex hull by
$\ch(A)$. Then by an \emph{edge} of $\ch(A)$ we mean a closed
segment $I$ of the straight line connecting two points of the
unit circle such that $I$ is contained in the boundary $\bd(\ch(A))$
of $\ch(A)$. By an \emph{edge} of a $\sim$-class we mean an edge of
the convex hull of that class.

\begin{dfn}[Laminations and dynamics]\label{d:si-inv-lam}
A laminational equivalence relation $\sim$ is ($\si_d$-){\em
in\-va\-riant} if:

\noindent (D1) $\sim$ is {\em forward invariant}: for a class
$\mathbf{g}$, the set $\si_d(\mathbf{g})$ is a class too;

\noindent (D2) for any $\sim$-class $\mathbf{g}$, the map
$\si_d:\mathbf{g}\to\si_d(\mathbf{g})$ extends to $\uc$ as an
orientation preserving covering map such that $\mathbf{g}$ is the full
preimage of $\si_d(\mathbf{g})$ under this covering map.

\end{dfn}

Again, if this does not cause ambiguity, we will simply talk about
\emph{invariant} laminational equivalence relations.

Definition~\ref{d:si-inv-lam} (D2) has an equivalent version. Given a
closed set $Q\subset \uc$, a (positively oriented) {\em hole} $(a, b)$
of $Q$ (or of $\ch(Q)$) is a component of $\uc\sm Q$. Then (D2) is
equivalent to the fact that for a $\sim$-class $\mathbf{g}$ either
$\si_d(\mathbf{g})$ is a point or for each positively oriented hole
$(a, b)$ of $\mathbf{g}$ the positively oriented arc $(\si_d(a),
\si_d(b))$ is a hole of $\si_d(\mathbf{g})$. From now on, we assume
that, unless stated otherwise, $\sim$ is a $\si_d$-invariant
laminational equivalence relation.

Given $\sim$, consider the \emph{topological Julia set}
$\uc/\sim=J_\sim$ and the \emph{topological polynomial}
$f_\sim:J_\sim\to J_\sim$ induced by $\si_d$. Since $\uc\subset \C$, we
can use Moore's Theorem to embed $J_\sim$ into $\C$ and then to extend
the quotient map $\psi_\sim:\uc\to J_\sim$ to a map $\psi_\sim:\C\to
\C$ with the only non-trivial fibers being the convex hulls of
non-degenerate $\sim$-classes. A \emph{Fatou domain} of $J_\sim$ (or of
$f_\sim$) is a bounded component of $\C\sm J_\sim$. If $U$ is a
periodic Fatou domain of $f_\sim$ of period $n$, then
$f^n_\sim|_{\bd(U)}$ is either conjugate to an irrational rotation of
$\uc$ or to $\si_k$ with some $1<k$, cf. \cite{bl02}. In the case of
irrational rotation, $U$ is called a \emph{Siegel domain}.
The complement of the unbounded component of $\C\sm J_\sim$ is called
the \emph{filled-in topological Julia set} and is denoted by $K_\sim$.
Equivalently, $K_\sim$ is the union of $J_\sim$ and its bounded Fatou
domains. If the laminational equivalence relation $\sim$ is fixed, we
may omit $\sim$ from the notation. By default, we consider $f_\sim$ as
a self-mapping of $J_\sim$. For a collection $\mathcal R$ of sets,
denote the union of all sets from $\mathcal R$ by $\mathcal R^+$.

\begin{dfn}[Leaves]\label{d:geola}
If $A$ is a $\sim$-class, call an edge $\ol{ab}$ of $\ch(A)$ a \emph{leaf of $\sim$}.
All points of $\uc$ are also called (\emph{degenerate}) leaves of $\sim$.
\end{dfn}

The family of all leaves of $\sim$ is closed (the limit of a sequence of leaves of $\sim$ is a leaf of $\sim$);
the union of all leaves of $\sim$ is a continuum.
For any subset $X\subset\disk$ with the property $X=\ch(X\cap\uc)$, we set $\si_d(X)=\ch(\si_d(X\cap\uc))$.
In particular, for any leaf $\ell$ of $\sim$, the set $\si_d(\ell)$ is a (possibly degenerate) leaf.

\subsection{Geodesic laminations}\label{ss:geol}
Assume that $\sim$ is a $\si_d$-invariant laminational equivalence relation.

\begin{dfn}\label{d:q-geolam}
The set $\lam_\sim$ of all leaves of $\sim$ is called the \emph{geodesic lamination generated by $\sim$}.
\end{dfn}

Thurston studied collections of chords in $\disk$ with similar properties to those of $\lam_\sim$.

\begin{dfn}[Geodesic laminations, \rm{cf.} \cite{thu85}]\label{d:geolam}
Two distinct chords in $\cdisk$ are said to be \emph{unlinked} if they
meet at most in a common endpoint; otherwise they are said to be
\emph{linked}, or to \emph{cross} each other.
A \emph{geodesic pre-la\-mi\-na\-tion} $\lam$ is a set of (possibly degenerate) chords
in $\ol{\disk}$ such that any two distinct chords from $\lam$ are unlinked.
A geodesic pre-lamination $\lam$ is called a \emph{geodesic lamination} if all points
of $\uc$ are elements of $\lam$, and $\lam^+$ is closed.
We will sometimes use the abbreviation \emph{geolamination} for a geodesic lamination.
Elements of $\lam$ are called \emph{leaves} of $\lam$. By a \emph{degenerate} leaf
(chord) we mean a singleton in $\uc$. The continuum $\lam^+\subset
\cdisk$ is called the \emph{solid} of $\lam$.
Let $\lam$ be a geodesic lamination. The closure in $\C$ of a non-empty
component of $\disk\sm \lam^+$ is called a \emph{gap} of $\lam$.
If $G$ is a gap or a leaf, call the set $G'=\uc\cap G$ the \emph{basis of $G$}.
A gap is said to be \emph{finite $($infinite, countable, uncountable$)$} if its basis is finite (infinite, countable,
uncountable).
Uncountable gaps are also called \emph{Fatou gaps}.
Points of $G'$ are called \emph{vertices} of $G$.
Geolaminations of the form $\lam_\sim$, where $\sim$ is a laminational equivalence relation, are called \emph{q-geolaminations} (``q'' from ``e\textbf{q}uivalence'').
\end{dfn}

Let us discuss geodesic laminations in the dynamical context.
A chord is called ($\si_d$-)\emph{critical} if its endpoints have the same image under $\si_d$.
If it does not cause ambiguity, we will simply talk about \emph{critical} chords.

It is essential to associate to each polynomial with connected Julia set a corresponding laminational equivalence relation (we have already done it in the case when the Julia sets is locally connected).
By Kiwi \cite{kiwi97}, this can also be done it $P$ is a polynomial with connected Julia set and such that all its periodic points are repelling (recall that such polynomials are said to be \emph{dendritic}).
Here is an important result proven by Kiwi in \cite{kiwi97}.

\begin{thm}
\label{t:kiwi-dendr} Suppose that a polynomial $P$ with connected Julia
set $J=J(P)$ has no Siegel or Cremer periodic points. Then there exist
a laminational equivalence relation $\sim_P$, the corresponding
topological polynomial $f_{\sim_P}:J_{\sim_P}\to J_{\sim_P}$ which
acts on its
topological Julia set, and a monotone semiconjugacy $\ph_P:J\to
J_{\sim_P}$. The semiconjugacy $\ph_P$ is one-to-one on all
$($pre$)$periodic points of $P$ belonging to $J$. If all periodic
points of $P$ are repelling, then $J_{\sim_P}$ is a dendrite.
\end{thm}

In what follows, denote by $\dc$ the space of all polynomials with
connected Julia sets and only repelling periodic points. Let $\dc_d$ be
the space of all such polynomials of degree $d$.
Then Theorem~\ref{t:kiwi-dendr} implies that the following diagram is commutative.

\[\label{eq:commdiag}
\dgARROWLENGTH=5em
\dgARROWPARTS=8
\begin{diagram}
\node{J(P)} \arrow{e,t}{P|_{J(P)}} \arrow{se,t}{\varphi_P}
    \node{J(P)} \arrow{se,b,1}{\varphi_P}
    \node{\mathbb S^1} \arrow{e,t}{\sigma_d} \arrow{sw,b,1}{\psi_{\sim_P}}
    \node{\mathbb S^1} \arrow{sw,t}{\psi_{\sim_P}}\\
\node[2]{J_{\sim_P}} \arrow{e,b}{f_{\sim_P}|_{J_{\sim_P}}} \node{J_{\sim_P}}
\end{diagram}
\]

The notion of \emph{sibling invariant laminations} introduced below is
slightly different from the original notion of \emph{invariant
laminations} in the sense of Thurston. However, sibling invariant
laminations form a closed set and include all $q$-laminations. Thus,
for all our purposes, it will suffice to consider sibling invariant
laminations only.
Some advantage of working with sibling $\si_d$-invariant geodesic
laminations is that they are defined through properties of their
leaves; gaps are not involved in the definition. It was shown in \cite{bmov13}
that all sibling invariant laminations are also invariant in the sense of Thurston \cite{thu85}.
In particular for any gap $G$ of a sibling invariant $\lam$ the set
$\si_d(G)$ is either a point, or a leaf of $\lam$, or a gap of $\lam$.

\begin{dfn}\label{d:siblinv}
A geodesic lamination $\lam$ is \emph{sibling $\si_d$-invariant} provided that:
\begin{enumerate}
\item for each $\ell\in\lam$, we have $\si_d(\ell)\in\lam$,
\item \label{2}for each $\ell\in\lam$ there exists $\ell'\in\lam$
    so that $\si_d(\ell')=\ell$.
\item \label{3} for each $\ell\in\lam$ so that $\si_d(\ell)=\ell'$
    is a non-degenerate leaf, there exist d {\bf disjoint} leaves
    $\ell_1,\dots,\ell_d$ in $\lam$ so that $\ell=\ell_1$ and
    $\si_d(\ell_i)=\ell'$ for all $i=1,\dots,d$.
\end{enumerate}
\end{dfn}

Let us list a few properties of sibling $\si_d$-invariant geodesic
laminations.

\begin{thm}[\cite{bmov13}]\label{t:siblinv}
The space of all sibling $\si_d$-invariant geodesic laminations is compact.
All geodesic laminations generated by $\si_d$-invariant laminational equivalence relations are sibling $\si_d$-invariant.
\end{thm}

For brevity, in what follows instead of ``sibling $\si_d$-invariant geodesic laminations'' we will say ``$\si_d$-invariant geodesic laminations''.
Below, we talk interchangeably about leaves (gaps) of $\sim$ or of $\lam_\sim$.
Let us now discuss gaps in the context of $\si_d$-invariant laminational equivalence relations and geodesic laminations.

\begin{dfn}[Critical gaps]\label{d:crit}
A gap\, $G$ of a geodesic lamination is called \emph{($\si_d$-)critical} if for each $y\in \si_d(G')$ the set
$\si_d^{-1}(y)\cap G'$ consists of at least $2$ points.
If it does not cause ambiguity, we talk about \emph{critical} gaps.
\end{dfn}

\begin{dfn}[Periodic and (pre)periodic gaps]\label{d:gaps-i}
Let $G$ be a gap of an invariant geodesic lamination $\lam$.
If the map $\si_d$ restricted to $G'$ extends to $\bd(G)$ as a composition of a monotone map and a covering map of some degree $m$, then $m$ is called the \emph{degree} of $\si_d|_G$.
A gap/leaf $U$ of $\lam_\sim$ is said
to be \emph{{\rm(}pre{\rm)}periodic} of period $k$ if
$\si_d^{m+k}(U')=\si_d^m(U')$ for some $m\ge 0$, $k>0$; if $m, k$ are
chosen to be minimal, then if $m>0$, $U$ is called \emph{preperiodic},
and, if $m=0$, then $U$ is called \emph{periodic $($of period $k)$}. If
the period of $G$ is $1$, then $G$ is said to be \emph{invariant}.
Define \emph{precritical} and \emph{{\rm(}pre{\rm)}critical} objects
similarly to (pre)periodic and preperiodic objects defined above.
\end{dfn}

Consider infinite periodic gaps of 
$\si_d$-invariant geodesic laminations.
There are three types of such gaps: \emph{caterpillar} gaps, \emph{Siegel} gaps, and \emph{Fatou gaps
of degree greater than one}.
Observe that, by \cite{kiw02}, infinite gaps are eventually mapped onto periodic infinite gaps.
First we state (without a proof) a well-known folklore lemma about the edges of preperiodic (in particular, infinite) gaps
(see, e.g., Lemma 2.28 \cite{bopt17}).

\begin{lem}\label{l:edges}
Any edge of a (pre)periodic gap is either (pre)periodic or (pre)critical.
\end{lem}

Let us now classify infinite gaps.

\begin{dfn}\label{d:caterpillar} An infinite  gap $G$ is said
to be a \emph{caterpillar} gap if its basis $G'$ is countable.
\end{dfn}

As as an example, consider a periodic gap $Q$ such that:
\begin{itemize}
 \item  The boundary of $Q$ consists of a periodic leaf
     $\ell_0=\ol{xy}$ of period $k$, a critical leaf
     $\ell_{-1}=\ol{yz}$ concatenated to it, and a countable
     concatenation of leaves $\ell_{-n}$ accumulating at $x$ (the
     leaf $\ell_{-r-1}$ is concatenated to the leaf $\ell_{-r}$,
     for every $r=1$, 2, $\dots$).
\item We have $\si^k(x)=x$, $\si^k(\{y, z\})=\{y\}$, and $\si^k$
    maps each $\ell_{-r-1}$ to $\ell_{-r}$ (all leaves are shifted
    by one towards $\ell_0$ except for $\ell_0$, which maps to
    itself, and $\ell_{-1}$, which collapses to the point $y$).
\end{itemize}

The description of $\si_3$-invariant caterpillar gaps is in
\cite{bopt13}. In general, the fact that the basis $G'$ of a
caterpillar gap $G$ is countable implies that there are lots of
concatenated edges of $G$. Other properties of caterpillar gaps can
be found in Lemma~\ref{l:cater}.

\begin{lem}[Lemma 1.15 \cite{bopt16}]\label{l:cater}
Let $G$ be a caterpillar gap of period $k$.
Then the degree of $\si_d^k|_{\bd(G)}$ is one, and $G'$ contains some periodic points of period $k$.
\end{lem}

\begin{dfn}\label{d:siegel}
A periodic Fatou gap $G$ of period $n$ is said to be a periodic \emph{Siegel} gap if the degree of $\si_d^n|_G$ is $1$, and the basis $G'$ of $G$ is uncountable.
\end{dfn}

The next lemma is well known.
A part of it  was actually proven in the proof of  \cite[Lemma~\ref{l:cater}]{bopt16}.

\begin{lem}\label{l:siegel}
Let $G$ be a Siegel gap of period $n$. Then the map $\si_d^n|_{\bd(G)}$
is monotonically semiconjugate to an irrational circle rotation and
contains no periodic points. A periodic Siegel gap must have at least
one image that has a critical edge.
\end{lem}

The following definition completes our series of definitions.

\begin{dfn}\label{d:fatou}
A periodic \emph{Fatou gap is of degree $k>1$} if the degree of
$\si_d^n|_{\bd(G)}$ is $k>1$. If the degree of a Fatou gap $G$ is $2$,
then $G$ is said to be \emph{quadratic}.
\end{dfn}

The next lemma is well known.

\begin{lem}\label{l:fatou}
Let $G$ be a Fatou gap of period $n$ and of degree $k>1$. Then the map
$\si_d^n|_{\bd(G)}$ is monotonically semiconjugate to $\si_k$.
\end{lem}

\section{Limit geodesic laminations and their properties}\label{s:limitg}
In this section, we deal with properties of limits of $\si_d$-invariant $q$-laminations.
For the most part, the results are obtained in \cite{bopt16}.

Fix a degree $d$.
In lemmas below, we assume that a sequence of $\si_d$-invariant $q$-laminations $\lam_i$ converges to a $\si_d$-invariant geodesic lamination $\lam_\infty$.
By a \emph{strip} we mean a (open) part of the unit disk contained between two disjoint chords.
By a \emph{strip around a chord $\ell$} we mean a strip containing $\ell$.
In what follows, when talking about convergence of leaves/gaps, closeness of leaves/gaps, and closures of families of
geodesic laminations, we always use the Hausdorff metric.

\begin{dfn}\label{d:closuq}
Let $\laq_d$ be the family of all $\si_d$-invariant geodesic $q$-laminations.
We will write $\ol{\laq_d}$ for the closure of $\laq_d$.
\end{dfn}

Even though we state below a few general results, we mostly concentrate on periodic objects of limit geodesic laminations.

\begin{lem}[Lemma 2.2 \cite{bopt16}]\label{l:limleaf1}
Let $\ell$ be a periodic leaf of $\lam\in \ol{\laq_d}$.
If $\hlam\in \laq_d$ is sufficiently close to $\lam$, then any leaf of any $\hlam$ sufficiently close to $\ell$ is either equal to $\ell$ or disjoint from $\ell$.
\end{lem}

Definition~\ref{d:rigileaf} introduces the concept of rigidity.

\begin{dfn}\label{d:rigileaf}
A leaf/gap $G$ of $\lam$ is \emph{rigid} if any $q$-lamination close to $\lam$ has $G$ as its leaf/gap.
\end{dfn}

A series of lemmas proved in \cite{bopt16} studies rigidity of periodic leaves/gaps of laminations from $\ol{\laq_d}$. These are combinatorial counterparts of the fact that repelling periodic points survive under small deformations of complex polynomials.
Observe that (periodic) leaves of geodesic laminations are either edges of gaps or limits of other leaves.
By a ($\si_d$-)\emph{collapsing polygon} we mean a polygon $P$, whose edges map under $\si_d$ to the same non-degenerate chord.
Thus, if a point makes a circuit around $P$, its $\si_d$-image moves back and forth along the same non-degenerate chord; as before,
if it does not cause ambiguity, we simply talk about \emph{collapsing polygons}.
When we say that $Q$ is a \emph{collapsing polygon of a geodesic lamination $\lam$}, we mean
that all edges of $Q$ are leaves of $\lam$; we also say that $\lam$ \emph{contains a collapsing polygon $Q$}. However, this does not imply that $Q$ is a gap of $\lam$ as $Q$ might be further subdivided by
leaves of $\lam$ inside $Q$.
Let us now quote some results of \cite{bopt16} concerning (pre)\-pe\-riodic leaves and gaps.

\begin{lem}[Lemma 2.5 - 2.10 \cite{bopt16}]\label{l:per-rigid}
Let $\lam\in \ol{\laq_d}$.
If $\hell\in \lam$ is a non-degenerate rigid leaf, a leaf $\ell\in \lam$ is such that $\si^k_d(\ell)=\hell$ for some
$k\ge 0$, and no leaf $\ell$, $\si_d(\ell)$, $\dots$, $\si^{k-1}(\ell)$ is contained in a collapsing polygon of $\lam$, then $\ell$ is rigid.
Also, the following objects are rigid:
\begin{enumerate}
 \item periodic leaves that are not edges of collapsing polygons;
 \item finite periodic gaps;
 \item $($pre$)$periodic leaves of a gap that eventually maps to a periodic gap;
 \item finite gaps that eventually map onto periodic gaps;
 \item periodic Fatou gaps whose images have no critical edges.
\end{enumerate}
\end{lem}

Using these results and other tools, we characterize all $\si_2$-invariant limit laminations.
Each such lamination $\lam$ can be described as a specific modification of an appropriate geodesic lamination $\lam^q$ from $\laq_2$.

\begin{dfn}\label{d:coexist}
Two geodesic laminations \emph{coexist} if their union is a lamination.
\end{dfn}

This notion was used in \cite{bopt13}.
If two geodesic laminations coexist, then a leaf of one lamination is either also a leaf of the other lamination or is located in a gap of the other geodesic lamination.

\begin{dfn}\label{d:hyperb}
A $\si_2$-invariant geodesic lamination is called \emph{hyperbolic} if it has a periodic Fatou gap of degree two.
\end{dfn}


Clearly, if a $\si_2$-invariant geodesic lamination $\lam$ has a periodic Fatou gap $U$ of
period $n$ and of degree greater than one, then the degree of $\si_2|_{\bd(U)}$ is two.
By \cite{thu85}, there is a unique edge $M(\lam)$ of $U$ with $\si_2^n(M(\lam))=M(\lam)$;
$M(\lam)$ is such that either all leaves $M(\lam)$, $\dots$, $\si_2^{n-1}(M(\lam))$ are
pairwise disjoint, or their union can be broken down into several gaps permuted by $\si_2$, in each of which
edges are ``rotated" by the appropriate power of $\si_2$, or $n=2k$ and
$\si_2^k$ flips $M(\lam)$ on top of itself while all leaves $M(\lam)$, $\dots$, $\si_2^{k-1}(M(\lam))$ are pairwise disjoint.
In fact, $M(\lam)$ and its sibling $M'(\lam)$ are the two \emph{majors} of $\lam$ while
$\si_2(M(\lam))=\si_2(M'(\lam))=m(\lam)$ is the \emph{minor} of $\lam$ \cite{thu85}
(recall that a \emph{major} of a $\si_2$-invariant geodesic lamination is the longest leaf of $\lam$).
Any $\si_2$-invariant hyperbolic geodesic lamination $\lam$ is actually
a geodesic lamination $\lam_\sim$ generated by the appropriate
\emph{hyperbolic $\si_2$-invariant laminational equivalence relation}
$\sim$ so that the topological polynomial $f_\sim$ considered on the
entire complex plane is topologically conjugate to a complex quadratic
\emph{hyperbolic} polynomial; this justifies our terminology.

\begin{dfn}\label{d:criset}
A \emph{critical set} $\cri(\lam)$ of a $\si_2$-invariant geodesic
lamination $\lam$ is either a critical leaf or a gap $G$ with
$\si_2|_{\bd(G)}$ of degree two.
\end{dfn}

A $\si_2$-invariant $q$-lamination $\lam$ either has a finite critical set (a critical leaf, or a finite critical gap) or is hyperbolic.
In both cases, the critical set is unique.

\begin{dfn}
By a \emph{generalized critical quadrilateral} $Q$ we mean either a 4-gon whose $\si_2$-image is a leaf, or a critical
leaf (whose image is a point). A \emph{collapsing \ql} is a generalized
critical \ql{} with four distinct vertices.
\end{dfn}

If $\cri(\lam)$ is a generalized critical \ql{}, then
$\si_2(\cri(\lam))=m(\lam)$. The notion of generalized critical
quadrilateral was used in \cite{bopt17} in our study of geodesic
laminations of arbitrary degree.

Theorem~\ref{t:limitlam} describes all geodesic laminations from
$\olq$. A periodic leaf $\n$ such that the period of its endpoints is
$k$ and all leaves $\n,$ $\si_2(\n),$ $\dots,$ $\si_2^{k-1}(\n)$ are
pairwise disjoint, is said to be a \emph{fixed return} periodic leaf.


\begin{thm}[Theorem 3.8 \cite{bopt16}]\label{t:limitlam}
A geodesic lamination $\lam$ belongs to $\olq$ if and only if there exists a unique maximal $q$-lamination $\lam^q$ coexisting with $\lam$ and such that either $\lam=\lam^q$ or $\cri(\lam)\subset \cri(\lam^q)$ is a
generalized critical quadrilateral, and exactly one of the following holds.
\begin{enumerate}
\item The critical set $\cri(\lam^q)$ is finite, and $\cri(\lam)$ is the convex hull of two edges or vertices of\, $\cri(\lam^q)$ with the same $\si_2$-image;
\item the lamination $\lam^q$ is hyperbolic with a critical Fatou gap $\cri(\lam)$ of period $n$, and exactly one of the following holds:
\begin{enumerate}
\item the set $\cri(\lam)=\ol{ab}$ is a critical leaf with a periodic
    endpoint of period $n$, and $\lam$ contains exactly two
    $\si_2^n$-pullbacks of $\ol{ab}$ that intersect $\ol{ab}$
    (one of these pullbacks shares an endpoint $a$ with
    $\ol{ab}$, and the other one shares an endpoint $b$ with
    $\ol{ab}$).
\item the critical set $\cri(\lam)$ is a collapsing quadrilateral, and $m(\lam)$ is a fixed return periodic leaf.
\end{enumerate}
\end{enumerate}
Thus, any $\si_2$-invariant $q$-lamination corresponds to finitely many geodesic laminations from $\olq$, and the union of all of their minors is connected.
\end{thm}

Given a geodesic lamination $\lam$, let $\lam^q$ denote the $\si_2$-invariant $q$-lamination associated with $\lam$ defined in Theorem~\ref{t:limitlam}.

To interpret the Mandelbrot set as a quotient of $\olq$, we define a certain equivalence relation on $\olq$, cf. \cite{bopt16}.

\begin{dfn}\label{d:minor}
Suppose that $\lam'$, $\lam''\in \olq$.
Then the geodesic laminations $\lam_1$ and $\lam_2$ are said to be \emph{minor equivalent} if there exists a
finite collection of geodesic laminations $\lam_1=\lam'$, $\lam_2$, $\dots$, $\lam_k=\lam''$ from $\olq$ such that for each $i$ with $1\le i\le k-1$, the minors $m(\lam_i)$ and $m(\lam_{i+1})$ of the geodesic
laminations $\lam_i$ and $\lam_{i+1}$ are non-disjoint.
\end{dfn}

The last result of \cite{bopt16} is the following theorem.
Let $\psi:\olq\to \uc/\qml$ be the map which associates to each geodesic
lamination $\lam\in \olq$ the $\qml$-class of the endpoints of the
minor $m(\lam)$ of $\lam$.

\begin{thm}[Theorem 3.10 \cite{bopt16}]\label{t:main}
The map $\psi:\olq\to \uc/\qml$ is continuous.
The partition of $\olq$ into classes of minor equivalence is upper semi-continuous, and the quotient space of $\olq$ with respect to the minor equivalence is homeomorphic to $\uc/\qml$.
\end{thm}

One can understand Theorem~\ref{t:main} as follows.
For every geodesic lamination $\lam$ define its \emph{minor set} as the image of its
critical set unless $\lam$ is hyperbolic in which case we call
$m(\lam)$ the \emph{minor set} of $\lam$. Then $\psi$ can be viewed as
the map associating to each class $A$ of minor equivalence in $\olq$
the minor set of the geodesic lamination $\lam^q$ that is the only $q$-lamination in $A$.
The minor set of $\lam^q$ is in fact the convex hull of the union of minors of all laminations in $A$.

We want to modify this result by considering the maximal perfect subset of $\olq$.
Observe that this subset consists of all non-isolated laminations.
Equivalently, we consider geodesic laminations which are limits of sequences of pairwise
distinct $\si_2$-invariant geodesic q-laminations.
Theorem~\ref{t:limitlam} implies Corollary~\ref{c:non-const-lim}.

\begin{cor}\label{c:non-const-lim}
A geodesic lamination $\lam\in \olq$ is non-isolated in $\olq$ if and
only if one of the following holds:
\begin{enumerate}
\item the critical set $\cri(\lam^q)$ is finite, and $\cri(\lam)$ is the convex hull of two
edges or vertices of $\cri(\lam^q)$ with the same $\si_2$-image;
\item the lamination $\lam^q$ is hyperbolic with a critical Fatou gap $\cri(\lam)$
    of period $n$, and exactly one of the following holds:
\begin{enumerate}
\item the set $\cri(\lam)=\ol{ab}$ is a critical leaf with a periodic
    endpoint of period $n$, and $\lam$ contains exactly two
    $\si_2^n$-pullbacks of $\ol{ab}$ that intersect $\ol{ab}$
    (one of these pullbacks shares an endpoint $a$ with
    $\ol{ab}$ and the other one shares an endpoint $b$ with
    $\ol{ab}$).
\item the set $\cri(\lam)$ is a collapsing quadrilateral, and $m(\lam)$ is a
    fixed return periodic leaf.
\end{enumerate}
\end{enumerate}
\end{cor}

In order to prove Corollary \ref{c:non-const-lim}, we need the following lemma.

\begin{lem}
\label{l:lam-uniq}
Suppose that $\lam$ is a $\si_2$-invariant $q$-lamination whose critical set is a generalized critical quadrilateral.
Then $\lam$ is the only $\si_2$-invariant geodesic lamination with critical set $\cri(\lam)$.
\end{lem}

The proof is based on Thurston's pullback construction \cite{thu85}.

\begin{proof}[Proof of Lemma \ref{l:lam-uniq}]
Indeed, properties of $\si_2$-invariant geodesic laminations imply that
pullbacks of $\cri(\lam)$ are well defined on each finite step; moreover, these pullbacks are all sets from $\lam$. Furthermore, the closure $\hlam$ of their entire family is a $\si_2$-invariant
geodesic lamination itself, and since $\lam$ is closed it follows that
$\hlam\subset \lam$.
We claim that $\hlam=\lam$.
Indeed, suppose otherwise.
Then $\hlam$ must contain a gap, say, $U$ that itself is the union of $s>1$ gaps of $\lam$ and,
therefore, \emph{$U$ contains leaves of $\lam$ \textbf{inside}}.
If $U$ is finite, it follows that there are non-disjoint finite gaps of $\lam$.
The latter is impossible as $\lam$ is a $q$-lamination.
Thus, $U$ is infinite.
Mapping $U$ forward several times, we may assume without loss of generality that $U$ is periodic of period $k$ (indeed, by \cite{kiw02}, all infinite gaps of geodesic laminations are (pre)periodic).

Consider several cases.
First suppose that $U$ is a caterpillar gap.
Then the critical leaf of $U$ (or of a gap in the forward orbit of $U$) must coincide with the critical set of $\lam$.
Therefore, $\lam$ has a critical leaf with a periodic endpoint, which is impossible for a $q$-lamination.

Now, suppose that $U$ is a Siegel gap. It is well-known (e.g., it
follows from Lemma~\ref{l:edges}) that all edges of $U$ are
(pre)critical and that, therefore, some image $\si_2^t(U)$ of $U$ has a
critical edge $\ell$; it then follows that $\cri(\lam)=\ell$, that all
edges of $U$ are pullbacks of $\ell$, and that under the map $\psi$
collapsing edges of $U$ to points any chord $\hell$ connecting vertices
of $U$ projects under $\psi$ to a non-trivial chord $\psi(\hell)$ of
the unit circle. Since $\psi$ semiconjugates $\si_2^k|_{\bd(U)}$ to an
irrational rotation, the chord $\psi(\hell)$ in the unit disk will
eventually intersect itself under the irrational rotation in question
which implies the same fact for the above chord $\hell\subset U$. We
see that $\hell$ cannot be a leaf of any lamination, a contradiction
with the above.

Finally, suppose that $\si_2^k|_{\bd(U)}$ is of degree $2$. Then some
image of $U$ is an infinite gap $V$ such that $\si_2|_{\bd(V)}$ has
degree two. On the other hand, $\cri(\hlam)=\cri(\lam)$ is a
generalized critical quadrilateral, a contradiction with the existence
of $V$. Hence this case is impossible either, and so
$\lam=\hlam=\lam^q$ is a unique geodesic lamination with critical set
$\cri(\lam)$.
\end{proof}

We are now ready to deduce Corollary \ref{c:non-const-lim} from Lemma \ref{l:lam-uniq}.

\begin{proof}[Proof of Corollary \ref{c:non-const-lim}]
By Theorem~\ref{t:limitlam}, if $\lam$ satisfies the conditions of the corollary, then $\lam\in \olq$.
Clearly, geodesic laminations described in (2)(a) and (2)(b) above do not belong to $\L^q_2$.
Since they belong to $\olq$, it follows that they are limits of sequences of pairwise distinct $\si_2$-invariant geodesic q-laminations.

Consider now case (1).
Then $\cri(\lam^q)$ is finite and $\cri(\lam)$ is the convex hull of two edges or vertices of $\cri(\lam^q)$ with the
same $\si_2$-image.
Suppose that $\cri(\lam^q)$ is a polygon with more than four vertices.
Then $\lam\neq \lam^q$ (in fact, $\lam\supsetneqq \lam^q$).
Hence $\lam\notin \L^q_2$, and, as above, $\lam$ is the limit of a sequence of pairwise distinct $\si_2$-invariant geodesic q-laminations.

Consider now the case when $\lam^q$ has a generalized quadrilateral as
its critical set $\cri(\lam^q)$. In this case it may happen that $\lam$
has a critical leaf that is a diagonal of a quadrilateral
$\cri(\lam^q)$ so that $\lam\neq \lam^q$; as before this implies that
$\lam$ is the limit of a sequence of pairwise distinct
$\si_2$-invariant geodesic laminations.

It remains to consider the case when $\lam=\lam^q$ is generated by an equivalence relation $\sim$ and has a critical set $\cri(\lam)$ that is either a critical quadrilateral or a critical leaf.
Let us show that then $\lam$ is the limit of a non-constant sequence of $q$-laminations.
By Lemma \ref{l:lam-uniq}, the lamination $\lam$ is the unique $\si_2$-invariant geodesic
lamination with critical set $\cri(\lam)$.
Now, the fact that $\lam$ is the limit of a sequence of pairwise distinct $q$-laminations follows from the uniqueness of
$\lam$ and the fact that, due to well-known properties of the combinatorial Mandelbrot set, there is a sequence of geodesic laminations $\lam_i$ with critical sets $\cri(\lam_i)\to \cri(\lam)$ (recall that we are considering the case when $\cri(\lam)$ is a generalized quadrilateral).
This completes the proof of the corollary.
\end{proof}

Thus, isolated geodesic laminations in $\olq$ are the following:
\begin{enumerate}
 \item dendritic geodesic laminations with critical sets that have more than four vertices,
 \item hyperbolic laminations.
\end{enumerate}
If we remove them from $\olq$, we will obtain the closed space $\L^l_2\subset \olq$ of all $\si_2$-invariant geodesic laminations that are non-isolated in $\olq$.
The minor equivalence on $\L^l_2$ is defined in the same way as the minor equivalence on $\olq$:
two laminations are \emph{minor equivalent} if their minors can be connected by a chain of non-disjoint minors.
However, we only consider minors of laminations from $\L^l_2$.
Therefore, the minor equivalence on $\L^l_2$ is not a restriction of the minor equivalence on $\olq$.
In particular, some classes of minor equivalence on $\L^l_2$ are slightly different from the restrictions of the corresponding classes of minor equivalence on $\olq$.
Consider several cases; the analysis below follows from Theorem~\ref{t:limitlam} and Corollary~\ref{c:non-const-lim}.

(1) Take a dendritic geodesic lamination $\lam$ generated by a laminational equivalence relation $\sim$ such that $\cri(\lam)$ has more than four vertices.
Then there are several geodesic laminations in $\L^l_2$ with critical sets being generalized critical quadrilaterals in $\cri(\lam)$.
These laminations form one class $A$ of the minor equivalence in $\L^l_2$.
However, unlike for $\olq$, the geodesic lamination $\lam$ itself does not belong to $\L^l_2$ and therefore is
not included into $A$.
Still, it follows that the convex hull of the union of all minors of laminations in $A$ is the same for $\L^l_2$ and for $\olq$.

(2) Now let $\lam$ be a dendritic geodesic lamination such that $\cri(\lam)$ is either a quadrilateral or a critical leaf. Then, by Corollary~\ref{c:non-const-lim}, we have $\lam\in \L^l_2$.
Hence, in this case, the corresponding class of minor equivalence in $\L^l_2$ consists of $\lam$ itself and two geodesic laminations obtained by inserting a critical diagonal in $\cri(\lam)$ and then pulling it back.
In other words, this class coincides with the corresponding class in $\olq$.
Of course, as in case (1), the convex hull of the union of minors remains the same as for $\olq$.

(3) In the case of a Siegel geodesic lamination $\lam$, we have that $\lam$ belongs to both $\L^l_2$ and $\olq$.
The corresponding class of the minor equivalence then consists of $\lam$ only.

(4) Consider a hyperbolic geodesic lamination $\lam$ with a critical gap $U$ of period $n$ such that the unique edge $M$ of $U$ of period $n$ is a fixed return leaf.
Then $\lam$ itself does not belong to $\L^l_2$.
However, three closely related geodesic laminations form a class of minor equivalence.
Two of them have critical leaves with endpoints at endpoints of $M$.
The third one has a collapsing quadrilateral based on $M$.
This yields the same convex hull of the union of minors as before in case of $\olq$.

(5) Finally, consider a hyperbolic geodesic lamination $\lam$ with a
critical gap $U$ of period $n$ such that the unique edge $M=\ol{ab}$ of
$U$ of period $n$ is \emph{not} a fixed return leaf. It follows that
neither $\lam$ nor the geodesic lamination with a collapsing
quadrilateral based on $M$ belong to $\L^l_2$. Therefore, there are
\emph{\textbf{two non-equivalent} geodesic laminations with critical
leaves $\ell_a$ and $\ell_b$ with endpoints $a$ and $b$, respectively} that can be
associated with $\lam$, and so there are \emph{\textbf{two classes of
minor equivalence}, generated by $\ell_a$ and $\ell_b$, respectively},
that can be associated with $\lam$.

Let $A$ be a class of minor equivalence in $\L^l_2$.
Define $m(A)$ as the convex hull of the union of the corresponding minors.
The association $A\mapsto m(A)$ is similar to that made in \cite{bopt16} for $\olq$.
Let $A'$ be the minor equivalence class in $\olq$ containing $A$.
The analysis made above implies that, in cases $(1)$ -- $(4)$, we have $m(A)=m(A')$.
In cases (2) and (3), we have $A=A'$.
In cases (1) and (4), the class $A'$ consists of $A$ and the lamination $\lam^q$ generated by the corresponding laminational equivalence.

In case (5) the situation is different.
The two distinct classes of minor equivalence in $\L^l_2$ correspond to critical leaves $\ell_a$ and $\ell_b$ and give rise to sets $\{\si_2(a)\}$ and $\{\si_2(b)\}$.
These singletons replace the minor $m(\lam)=\ol{\si_2(a) \si_2(b)}$ that corresponds to $\lam$ in $\qml$.
In other words, the leaf $m(\lam)$ is erased from $\qml$ and replaced by its two endpoints.
This is what we informally referred to in the beginning of the paper as ``unpinching'' of the circle.
In the end, this yields a new parametric geodesic lamination $\qml^l$,
the corresponding laminational equivalence $\sim_{\qml^l}$,
and the corresponding quotient space $\M^l_2$.
Let $\psi^l:\L^l_2\to \uc/\qml^l$ be the quotient map.

\begin{thm}\label{t:main1}
The map $\psi^l:\L^l_2\to \uc/\qml^l$ is continuous.
Thus, the partition of $\L^l_2$ into classes of minor equivalence is upper semi-continuous,
and the quotient space of{} $\L^l_2$ with respect to the minor equivalence is homeomorphic to $\uc/\qml^l$.
\end{thm}

It follows that the corresponding parametric geodesic lamination $\qml^l$, taken as a subset of the space of all chords in $\cdisk$, is perfect.

The dendritic versions $\sim^{ld}$ and $\M^{ld}_2$ of $\sim^{l}$ and $\M^{l}_2$ were given in
the Introduction, so we refer the reader there; notice that the dendritic versions
here are well-defined as $\qml^l$ is perfect.

\section{Non-renormalizable geodesic laminations}\label{s:non-renorm}

In this short section we suggest two spaces of $\si_2$-invariant
geodesic laminations that can be obtained from $\M^c_2$ through a
process of ``unpinching'' (i.e., replacing some leaves of $\qml$ by
their endpoints). Both spaces are described by the same parametric
geodesic lamination $\qml^{nr}$ modifying Thurston's quadratic minor
lamination $\qml$ and such that infinite gaps of it are all pairwise
disjoint. As before, this geodesic lamination can be interpreted in two
different ways resulting in two distinct but related quotient spaces.
First, we may consider various vertices and edges of infinite gaps of
the $\qml^{nr}$ as distinct classes of the corresponding equivalence
relation; then infinite gaps of $\qml^{nr}$ become closures of bounded
complementary domains of the quotient space $\cdisk/\qml^{nr}=\M^{nr}_2$. Then we
also consider the dendritic version $\M^{nrd}_2$ of $\M^{nr}_2$. The main ideas
were already discussed in the Introduction.


Consider a maximal by inclusion \emph{canonical} copy $X\subsetneqq \M^c_2$ of the combinatorial Mandelbrot set $\M^c_2$.
In other words, consider a copy $X\subsetneqq \M^c_2$ of $\M^c_2$
that can be described as follows.
There exist a special \emph{ancestor laminational equivalence} $\sim_X$ and the associated $\si_2$-invariant
\emph{ancestor geodesic lamination} $\lam_X$ with the following properties.
The lamination $\lam_X$ is hyperbolic and has a critical Fatou gap $U_X$ of period, say, $n$.
Let $V_X=\si_2(U_X)$ be the postcritical gap of $\lam_X$.
The boundary $\bd(V_X)$ can be projected to $\uc$ by means of collapsing its edges.
This projection will be denoted by $\ph_X$.
It semiconjugates $\si_2^n|_{\bd(V_X)}$ and $\si_2$.
Moreover, if we consider $\qml\subset \cdisk$ and lift it to $V_X$ using $\ph_X$, we will
obtain \emph{exactly} the restriction of $\qml$ onto $V_X$; the quotient space of $V_X$ under this restriction is \emph{exactly} $X$.

This description applies to any canonical copy of $\M^c_2$ in $\M^c_2$; in fact, it can be viewed as
the definition of a canonical copy of $\M^c_2$ inside $\M^c_2$.
Since we only want maximal by inclusion canonical copies $X$ of $\M^c_2$ in $\M^c_2$, we need to specify properties
of $\lam_X$.
Well-known facts about the combinatorial Mandelbrot set imply that maximal gaps $V_X$ can be of two types.

(1) There is a countable family of such gaps, each of which is based on an edge of the Main Cardioid.
The corresponding geodesic lamination $\lam_X$ has a finite invariant rotational gap $G_X$ such that
$\si_2(G_X)=G_X$, and the map $\si_2|_{G_X}$ can be viewed as a ``combinatorial rotation''
(each edge of $G_X$ under $\si_2$ ``jumps'' over the same number of edges of $G_X$ in the positive direction).
Moreover, there is a cycle of Fatou gaps attached to $G_X$ at its edges.
The longest edge $M_X$ of $G_X$ is the major of $\lam_X$ at which the critical gap $U_X$ is attached to $G_X$.
The edge $\si_2(M_X)=m_X$ is the minor of $\lam_X$.
The gap $V_X=\si_2(U_X)$ is attached to $G_X$ at $m_X$.
As explained above, $X$ is a copy of $\M^c_2$ that can be viewed as a ``pinched gap'' $V_X$.
There are countably many such copies $X$ of $\M^c_2$ and, respectively, countably many gaps $V_X$ attached at edges of the Main Cardioid
to the Main Cardioid itself.
The remaining vertices of the Main Cardioid correspond to laminations with invariant Siegel gap.

(2) Otherwise, there are maximal by inclusion copies $X$ of $\M^c_2$ in $\M^c_2$ that are generated by gaps $V_X$ of different nature.
In these cases, $\lam_X$ is a hyperbolic lamination such that $U_X$ is a periodic critical Fatou gap
of $\lam_X$ of some period $n$ with a fixed return edge of period $n$.
Notice that these are only \emph{maximal by inclusion} Fatou gaps $U_X$
of some period
$n$ such that $U_X$ has a fixed return edge of period $n$; there exist other
hyperbolic geodesic laminations with similar periodic Fatou gaps but
we consider only maximal ones among them.

Overall, we have countably many sets $X$ and countably many gaps $V_X$.
All these gaps are pairwise disjoint. In order to construct a new
parametric geodesic lamination $\qml^{nr}$ (``nr'' stands for
``non-renormalizable'') we proceed as follows. First, in case (2) above
we replace sets $X$ by the corresponding gaps $V_X$. In other words, we
erase all minors inside gaps $V_X$. In $\qml$, the gap $V_X$ was
``pinched'' by a geodesic lamination that copies $\qml$ itself; now we ``unpinch''
$V_X$. In case (1) above we do almost the same with one exception. Indeed,
let $X$ be a canonical copy of $\M^c_2$ associated with a lamination $\sim_X$ and
a postcritical Fatou gap $V_X$ of $\lam_X$. Then we erase not only all minors
of $\qml$ inside $V_X$ but also the edge $m_X$ of $V_X$ which is an edge of the Main
Cardioid. This removes the ``barrier'' between $V_X$ and the Main Cardioid. Evidently,
it leads to a big gap which unites all gaps $V_X$ from case (1) above and all Siegel vertices
of the Main Cardioid into one new gap $\ca^{nr}$ that can be viewed as the central gap
of $\qml^{nr}$ replacing in this capacity the Main Cardioid.

From the point of view of $\si_2$-invariant geodesic laminations, vertices and edges
$\ell$ of $V_X$ represent geodesic laminations that can be constructed as follows.
Take, say, an edge $\ell$ of $V_X$.
Then take the convex hull of its full preimage in $U_X\cap\uc$;
this will produce a collapsing  quadrilateral $Q_\ell\subset U_X$.
Using Thurston's pullback construction, we can repeat this process infinitely many times.
Then take the closure of the resulting set of leaves.
This, according to \cite{thu85}, will be a $\si_2$-invariant geodesic lamination whose critical
set projects to a critical leaf under the map from $\bd(U_X)$ to $\uc$ collapsing of all edges of $U_X$.

Thus, again, the idea is to choose a specific class of geodesic
laminations (in this case, these are all renormalizable geodesic
laminations) and replace them by geodesic laminations that have either
critical leaves or collapsing quadrilaterals (projecting to critical
leaves under the canonical collapse of edges of the corresponding
postcritical Fatou gap). In other words, $\qml^{nr}$, $\sim^{nr}$, and the
corresponding quotient space $\M^{nr}_2$ describe the family of
laminations that can be characterized as follows. First, these are all
non-renormalizable $q$-laminations. Second, these are all geodesic
laminations obtained by inserting a critical leaf or a quadrilateral
projecting to a critical leaf under the collapse of individual edges in
a maximal, by inclusion, critical Fatou gap.

Each edge of the gap $V_X$ of period $n$ is an eventual pullback of a
unique edge of $V_X$ of period $n$, which is in fact the minor $m_X$ of
$\lam_X$. In $\M^{nr}_2$, edges of $V_X$ represent cutpoints of
$\M^{nr}_2$ belonging to the boundaries of bounded domains inside
$\M^{nr}_2$. These are exactly the places at which $V_X$ connects with
the rest of $M^{nr}_2$. Since all edges of $V_X$ are preimages of the
unique edge of $V_X$ of period $n$, it follows that $\M^{nr}_2$ grows
from $V_X$ (viewed as a parametric gap) at specific places
corresponding to specific laminations inside the critical gap $U_X$.
These are laminations that, after renormalization (that is, after
projecting $U_X$ to the entire circle), become laminations with a
critical leaf eventually mapped to the point of $\uc$ with argument
$0$.

All of the above applies to all gaps $V_X$ from case (2). Moreover, it
almost completely applies to the gaps $V_X$ from case (1). The only exception
to this is that the laminations from the Main Cardioid itself will not
have their minors represented as the edges of $\ca^{nr}$.

Finally, well-known facts about the combinatorial Mandelbrot set imply
that $\qml^{nr}$ is perfect. Thus the parametric geodesic lamination
$\qml^{nr}$ can in fact be interpreted in a different fashion giving
rise to the dendritic versions $\sim^{nrd}$ and $\M^{nrd}_2$ of
$\sim^{nr}$ and $\M^{nr}_2$. The main idea was explained in the
introduction. The space $\M^{nrd}_2$ represents the model for quotient
of the space of all $q$-laminations determined by the equivalence
relation under which two $q$-laminations $\lam_1, \lam_2$ are
identified in the following two cases: (1) is the oldest ancestor of
them is the empty lamination (this means that either lamination is a tuning of
a lamination from the Main Cadrioid), or (2) $\lam_1$ and $\lam_2$ have the same
oldest ancestor lamination which is not the empty lamination.

\bibliographystyle{amsalpha}

\end{document}